\documentclass[11pt,a4paper,reqno]{amsart}
\usepackage{mathrsfs}
\usepackage{syntonly}
\usepackage{amsmath}
\usepackage{amsthm}
\usepackage{amsfonts}
\usepackage{amssymb}
\usepackage{latexsym}
\usepackage{amscd,amssymb,amsopn,amsmath,amsthm,graphics,amsfonts,mathrsfs,accents,enumerate,verbatim,calc}
\usepackage[dvips]{graphicx}
\usepackage[colorlinks=true,linkcolor=blue,citecolor=blue]{hyperref}
\usepackage[all]{xy}

\date{}
\allowdisplaybreaks[4] \footskip=15pt
\renewcommand{\uppercasenonmath}[1]{}

\numberwithin{equation}{section} \theoremstyle{plain}
\newtheorem{lem}{Lemma}[section]
\newtheorem{cor}[lem]{Corollary}
\newtheorem{prop}[lem]{Proposition}
\newtheorem{thm}[lem]{Theorem}

\newtheorem{cond}[lem]{Condition}
\newtheorem{definition}[lem]{Definition}
\newtheorem{Ex}[lem]{Example}
\newtheorem{Quest}[lem]{Question}
\newtheorem{Property}[lem]{Property}
\newtheorem{Properties}[lem]{Properties}
\newtheorem{Subprops}{}[lem]
\newtheorem{Para}[lem]{}

\newtheorem{fact}[lem]{Fact}

\newtheorem{remark}[lem]{Remark}
\newtheorem{rem}[lem]{Remark}

\def\blue{\color{blue}}

\newtheorem*{ack*}{ACKNOWLEDGEMENTS}

\newcommand{\pf}{\noindent\begin {proof}}
\newcommand{\epf}{\end{proof}}

\newcommand{\X}{\mathcal{X}}

\newcommand{\Z}{\mathbb{Z}}

\begin{document}
\setlength{\voffset}{-0.8cm}

\begin{center}
{\large  \bf  Chains of model structures arising from cotorsion pairs on extriangulated categories}

\vspace{0.5cm}  Dandan Sun$^{a}$, Xiaoyan Yang$^{b}$, Dongdong Zhang$^{c}$\footnote{Corresponding authors. Xiaoyan Yang is supported by the National Natural Science Foundation of China (Grant No. 12571035). Dongdong Zhang is supported by National Natural Science Foundation of China (Grant No. 12571042). Panyue Zhou is supported by the National Natural Science Foundation of China (Grant No. 12371034) and by the Scientific Research Fund of Hunan Provincial Education Department (Grant No. 24A0221). Haiyan Zhu is supported by the the National Natural Science Foundation of China (Grant No. 12271481).}, Panyue Zhou$^{d{\blue *}}$ and Haiyan Zhu$^{a}$\\
\medskip

\hspace{-4mm}$^{a}$School of Mathematical Sciences, Zhejiang University of Technology, Hangzhou 310023, China\\
$^{b}$School of Science, Zhejiang University of Science and Technology, Hangzhou 310023, China\\
$^c$School of Mathematical Sciences, Zhejiang Normal University,
\small Jinhua 321004, China\\
$^d$School of Mathematics and Statistics, Changsha University of Science and Technology, \\[1mm]
Changsha 410114, China\\
E-mails: 13515251658@163.com, yangxy@zust.edu.cn, zdd@zjnu.cn, panyuezhou@163.com and hyzhu@zjut.edu.cn \\
\end{center}

\bigskip
\centerline { \bf  Abstract}
\medskip
\leftskip10truemm \rightskip10truemm
\noindent The main aim of this paper is to study chains of model structures arising from cotorsion pairs in extriangulated categories. Starting with a hereditary Hovey triple, we construct further hereditary Hovey triples whose homotopy categories  are equivalent under suitable completeness assumptions, thereby  refining results due to El Maaouy and Shao-Wang-Zhang. As an application, we consider objects of finite Gorenstein injective dimension with respect to a proper class of $\mathbb{E}$-triangles. Under mild set-theoretic assumptions, we obtain a chain of model structures whose homotopy categories are all triangulated equivalent to a common stable category. This recovers known results for Gorenstein injective modules and yields new examples in the derived category of a ring when the proper class is given by cohomological ghost triangles.
\\[2mm]

\leftskip0truemm \rightskip0truemm
\bigskip

\section { \bf Introduction}

Cotorsion pairs and their relation with model structures have been studied extensively. By a result of Hovey \cite{HCc}, suitable cotorsion pairs give rise to model structures. This correspondence has been extended to exact and extriangulated categories via Hovey triples (see \cite{Gillespie,G,NP}). Recall that extriangulated categories, introduced by Nakaoka and Palu \cite{NP}, provide a common generalization of exact and triangulated categories. In this setting, extensions, cotorsion pairs, and homological dimensions can be developed in close analogy with the classical cases.

Given a cotorsion pair, one may consider the associated subcategories of objects of finite homological dimension. These subcategories often induce new cotorsion pairs, and hence give rise to further model structures. In this way, one obtains chains of model structures indexed by the non-negative integers. Examples of this phenomenon include extendable cotorsion pairs \cite{SWZ}, modules of finite Gorenstein projective dimension \cite{GLZ}, modules of finite Gorenstein injective dimension \cite{El Maaouy2}, and modules of finite Gorenstein flat dimension \cite{El Maaouy1}.

On the other hand, $\xi$-injective and $\xi$-$\mathcal{G}$injective objects were introduced by Hu-Zhang-Zhou \cite{HZZ} with respect to a proper class $\xi$ of $\mathbb{E}$-triangles. These notions extend injective and Gorenstein injective objects to the setting of extriangulated categories. The subcategories consisting of objects of finite $\xi$-$\mathcal{G}$injective dimension form a chain indexed by the corresponding homological dimension. However, the relation between this chain and model structures is unclear in general.

The purpose of this paper is to investigate chains of model structures arising from cotorsion pairs in extriangulated categories, and apply the resulting theory to objects of finite $\xi$-$\mathcal{G}$injective dimension. Our arguments rely essentially on the structural properties of extriangulated categories.

To state the main results, we recall some notation and definitions.  Assume that $(\mathcal{B}, \mathbb{E}, \mathfrak{s})$ is an extriangulated category. Here $\mathcal{B}$ is an additive category, $\mathbb{E}: \mathcal{B}^{\rm op}\times \mathcal{B} \rightarrow {\rm Ab}$ is an additive bifunctor, and $\mathfrak{s}$ associates to each $\delta \in \mathbb{E}(C,A)$ a collection of $3$-term sequences whose end terms are $A$ and $C$, see \cite[Definition 2.10]{NP}. In the absence of ambiguity, we denote the quadruple simply by $\mathcal{B}:=(\mathcal{B}, \mathbb{E}, \mathfrak{s})$.

Let $(\mathcal{X}, \mathcal{Y})$ be a complete cotorsion pair in $\mathcal{B}$. We call $\mathcal{X}\cap\mathcal{Y}$ the \emph{core} of $(\mathcal{X}, \mathcal{Y})$. For each integer $n\geqslant 0$, we denote by $\mathcal{Y}_n$ the full subcategory of $\mathcal{B}$ consisting of all objects of $\mathcal{Y}$-coresolution dimension at most $n$. For brevity, we write $^\perp\mathcal{Y}_n$ in place of $^\perp(\mathcal{Y}_n)$.

Recall that a triple $(\mathcal{C}, \mathcal{W}, \mathcal{F})$ of $\mathcal{B}$ is called a {\it Hovey triple} if $(\mathcal{C}\cap \mathcal{W}, \mathcal{F})$ and $(\mathcal{C}, \mathcal{W}\cap \mathcal{F})$ are complete cotorsion pairs, and $\mathcal{W}$ is {\it thick} in $\mathcal{B}$, that is, $\mathcal{W}$ is closed under direct summands and satisfies the two-out-of-three property for conflations. A Hovey triple is called {\it hereditary} if the cotorsion pairs $(\mathcal{C}\cap \mathcal{W}, \mathcal{F})$ and $(\mathcal{C}, \mathcal{W}\cap \mathcal{F})$ are hereditary. The homotopy category associated with a Hovey triple $\mathcal{M}=(\mathcal{C}, \mathcal{W}, \mathcal{F})$ is denoted by ${\rm Ho}(\mathcal{M})$.

For a full additive subcategory $\mathcal{U}$ of an additive category $\mathcal{A}$, recall that the quotient category $\underline{\mathcal{A}}$ has the same objects as $\mathcal{A}$, and
$$
\textrm{Hom}_{\underline{\mathcal{A}}}(X,Y)
=
\textrm{Hom}_{\mathcal{A}}(X,Y)
/\textrm{Hom}_{\mathcal{A}}(X,\mathcal{U},Y),
$$
where $\textrm{Hom}_{\mathcal{A}}(X,\mathcal{U},Y)$ denotes the subgroup of $\textrm{Hom}_{\mathcal{A}}(X,Y)$ consisting of morphisms factoring through an object of $\mathcal{U}$. Then $\underline{\mathcal{A}}$ is an additive category.

Our first result provides a method for constructing chains of model structures from a Hovey triple, and shows that these model structures have equivalent homotopy categories under suitable assumptions. For notation not defined in the theorem, we refer the reader to Section \ref{Preliminaries}.

\begin{thm}\label{thm2} Let $(\mathcal{B}, \mathbb{E}, \mathfrak{s})$ be a weakly idempotent complete
 extriangulated category, and $\mathcal{M}=(\mathcal{C}, \mathcal{W}, \mathcal{F})$  a hereditary Hovey triple in $\mathcal{B}$. Assume that $(^\perp(\mathcal{W}\cap\mathcal{F})_n, (\mathcal{W}\cap\mathcal{F})_n)$ is a complete cotorsion pair for some non-negative integer $n$. Then

$(1)$ The pair $({^\perp(\mathcal{W}\cap \mathcal{F})}_n\cap\mathcal{W}, \mathcal{F}_n)$
is a complete and hereditary cotorsion pair with core $^\perp(\mathcal{W}\cap\mathcal{F})_n\cap (\mathcal{W}\cap\mathcal{F})_n$.

$(2)$ The triple $\mathcal{M}_n=({^\perp(\mathcal{W}\cap \mathcal{F})}_n, \mathcal{W}, \mathcal{F}_n)$ is a hereditary Hovey triple in $\mathcal{B}$, and ${^\perp(\mathcal{W}\cap \mathcal{F})}_n\cap\mathcal{F}_n$ is a Frobenius extriangulated category with the class of projective-injective objects ${^\perp(\mathcal{W}\cap \mathcal{F})}_n\cap(\mathcal{W}\cap \mathcal{F})_n$. Furthermore, we obtain triangle equivalences
$$\underline{{^\perp(\mathcal{W}\cap \mathcal{F})}_n\cap\mathcal{F}_n}\simeq {\rm Ho}(\mathcal{M}_n)= {\rm Ho}(\mathcal{M})\simeq \underline{\mathcal{C}\cap\mathcal{F}}.$$
\end{thm}
Now, we compare Theorem \ref{thm2} with some results in the literature.

\begin{itemize}
\item Theorem \ref{thm2} improves the result of Shao-Wang-Zhang in \cite[Theorem 3.9]{SWZ} by showing that the condition ``the exact category has enough projectives and enough injectives" is unnecessary (see Corollary \ref{cor3.4} and Remark \ref{rem:3.5}(1)).

  \item Theorem \ref{thm2} may be regarded as a refinement of the work of El Maaouy in \cite[Theorem A]{El Maaouy2} (see Remark \ref{rem:3.5}(2)). In that paper, the hereditary Hovey triple $({^\perp(\mathcal{W}\cap \mathcal{F})}_n, \mathcal{W}, \mathcal{F}_n)$ was obtained under the further hypotheses that $\mathcal{B}$ is an abelian category and $(\mathcal{C}, \mathcal{W}, \mathcal{F})$ is a hereditary Hovey triple on $\mathcal{B}$ for which:

(i) $(\mathcal{W},\mathcal{W}^\perp)$ is a complete cotorsion pair;

(ii) $(^\perp(\mathcal{W}\cap\mathcal{F})_k, (\mathcal{W}\cap\mathcal{F})_k)$ is a complete cotorsion pair for every $k\in\{0, 1, \cdots, n\}$.

\end{itemize}


For an arbitrary ring $R$, it is a well-established fact that the subcategory $\mathcal{GI}(R)$ of Gorenstein injective modules, together with the collection of short exact sequences whose three terms are all in $\mathcal{GI}(R)$, is a Frobenius category. Let $\mathrm{Mod}R$ denote the category of left $R$-modules and $\mathcal{I}(R)$ the subcategory of injective modules. This result admits a natural extension to an infinite sequence of subcategories:
$$\mathcal{GI}(R)={^\perp{\mathcal{I}(R)_{0}}\cap\mathcal{GI}(R)_{0}},~
{^\perp{\mathcal{I}(R)_{1}}\cap\mathcal{GI}(R)_{1}},~\cdots,{^\perp{\mathcal{I}(R)_{k}}\cap\mathcal{GI}(R)_{k}},\cdots.$$
Consequently, each associated stable category $\underline{\mathcal{I}(R)_{k} \cap \mathcal{GI}(R)_{k}}$ inherits a triangulated structure. The question of how this stable category evolves as the Gorenstein injective dimension increases does not yield an obvious answer when considered purely within the framework of stable categories. However, by analyzing these stable categories through the lens of model structures, El Maaouy provided a definitive resolution in \cite[Proposition 3.9]{El Maaouy2}, establishing the following triangle equivalences:
\begin{equation}\label{eq1.1}
\underline{\mathcal{GI}(R)}\simeq
\underline{^\perp\mathcal{I}(R)_{1} \cap \mathcal{GI}(R)_{1}}\simeq\cdots\simeq\underline{^\perp{\mathcal{I}(R)_{k}}\cap\mathcal{GI}(R)_{k}}\simeq\cdots. \end{equation}
The essential ingredient underpinning this proof is  that the complete cotorsion pair $(\mathrm{Mod}R,\mathcal{I}(R))$ always induces the completeness of cotorsion pairs $(^{\perp}\mathcal{I}(R)_{k},\mathcal{I}(R)_{k})$ for all integers $k > 0$.

Let $\xi$ be a proper class in an extriangulated category $(\mathcal{B}, \mathbb{E}, \mathfrak{s})$. By \cite[Theorem 3.2]{HZZ}, the triple $(\mathcal{B}, \mathbb{E}_\xi, \mathfrak{s}_\xi)$ carries a natural extriangulated structure. Following \cite{HZZ}, we denote by $\mathcal{GI}(\xi)$ the full subcategory of $\xi$-$\mathcal{G}$injective objects in $\mathcal{B}$, and by $\mathcal{I}(\xi)$ the full subcategory of $\xi$-injective objects in $\mathcal{B}$. It is known that $\mathcal{GI}(\xi)$ is a Frobenius extriangulated category, and hence its stable category $\underline{\mathcal{GI}(\xi)}$ is triangulated. Analogously to the triangle equivalences in {\rm (\ref{eq1.1})} for module categories, one may ask how the stable category behaves as the $\xi$-$\mathcal{G}$injective dimension of objects in $\mathcal{B}$ increases. To address this question, one is naturally led to determine when the complete cotorsion pair $(\mathcal{B},\mathcal{I}(\xi))$ induces the completeness of the cotorsion pairs $(^{\perp}\mathcal{I}(\xi)_{k},\mathcal{I}(\xi)_{k})$ for all integers $k>0$. At present, we do not know how to achieve this in general. Nevertheless, Proposition \ref{lem4.7} shows that this is indeed the case under the following assumptions: the pair $({^\bot}\mathcal{I}(\xi),\mathcal{I}(\xi))$ is cogenerated by a set $\mathcal{S}$ in the extriangulated category $(\mathcal{B}, \mathbb{E}_\xi, \mathfrak{s}_\xi)$, and $\mathcal{B}$ satisfies

(Ax1) Arbitrary transfinite compositions of $\xi$-inflations exist and remain $\xi$-inflations;

(Ax2) Every object of $\mathcal{B}$ is small relative to the class of all $\xi$-inflations.

\noindent These assumptions are inspired by the work of Saor$\acute{{\i}}$n--$\mathrm{\check{S}}\mathrm{\check{t}}\textrm{ov}\acute{{\i}}\mathrm{\check{c}}\textrm{ek}$ \cite{JSS}. For notations not defined in the proposition, we refer to Section \ref{model-sts-arsing-from-GI}.

As a consequence of Theorem \ref{thm2}, the following result provides sufficient conditions for constructing chains of model structures arising from objects of finite $\xi$-$\mathcal{G}$injective dimension. Moreover, it shows that enlarging the $\xi$-$\mathcal{G}$injective dimension does not alter the associated stable category.

\begin{thm}\label{thm4} Let $n$ be a non-negative integer, and let $(\mathcal{B}, \mathbb{E}_\xi, \mathfrak{s}_\xi)$ be a  weakly idempotent complete
 extriangulated category satisfying {\rm (Ax1)} and {\rm (Ax2)}. Assume that $(^\perp\mathcal{GI}(\xi), \mathcal{GI}(\xi))$ is a complete cotorsion pair in $(\mathcal{B}, \mathbb{E}_\xi, \mathfrak{s}_\xi)$ and  that the pair $({^\bot}\mathcal{I}(\xi),\mathcal{I}(\xi))$ is cogenerated by a set $\mathcal{S}$ in $(\mathcal{B}, \mathbb{E}_\xi, \mathfrak{s}_\xi)$. Then

$(1)$ The pair $(^\perp\mathcal{I}(\xi)_n\cap{^\perp\mathcal{GI}}(\xi), \mathcal{GI}(\xi)_n)$ is a complete hereditary cotorsion pair in $(\mathcal{B}, \mathbb{E}_\xi, \mathfrak{s}_\xi)$ with core $\mathcal{I}(\xi)_n\cap{^\perp\mathcal{I}}(\xi)_n$.

$(2)$ The triple $(^\perp\mathcal{I}(\xi)_n, {^\perp\mathcal{GI}(\xi)}, \mathcal{GI}(\xi)_n)$ is a hereditary Hovey triple in $(\mathcal{B}, \mathbb{E}_\xi, \mathfrak{s}_\xi)$; $^{\perp}\mathcal{I}(\xi)_n\cap\mathcal{GI}(\xi)_n$ is a Frobenius extriangulated category such that $^{\perp}\mathcal{I}(\xi)_n\cap\mathcal{I}(\xi)_n$ is its class of projective-injective objects; and the homotopy category is the stable category $\underline{^{\perp}\mathcal{I}(\xi)_n\cap\mathcal{GI}(\xi)_n}$, which is triangle equivalent to $\underline{\mathcal{GI}(\xi)}$.
\end{thm}

It is worth noting that Theorem \ref{thm4} not only extends El Maaouy's result \cite[Corollary 3.9]{El Maaouy2} on chains of model structures arising from modules of finite Gorenstein injective dimension, but also yields a new result for the derived category of a ring when the proper class is given by the cohomological ghost triangles introduced by Otake \cite{O} (see Corollary \ref{cor:4.12}).

While the work of $\mathrm{\check{S}}$aroch and $\mathrm{\check{S}}\mathrm{\check{t}}\textrm{ov}\acute{{\i}}\mathrm{\check{c}}\textrm{ek}$ \cite{SS} shows that $(^{\perp}\mathcal{GI}(R),\mathcal{GI}(R))$ is a complete cotorsion pair for any ring $R$, it remains unknown whether an analogous completeness result holds for the cotorsion pair $(^\perp\mathcal{GI}(\xi), \mathcal{GI}(\xi))$ in the extriangulated category $(\mathcal{B}, \mathbb{E}_\xi, \mathfrak{s}_\xi)$. Motivated by the elegant construction of Gao--Lu--Zhang for chains of model structures on exact categories arising from modules of Gorenstein projective dimension bounded by a non-negative integer (see \cite[Theorem 6.3]{GLZ}), we establish the following result.

\begin{thm}\label{thm4.15'} Let $m$ and $n$ non-negative integers with $m\leqslant n$, and let $(\mathcal{B}, \mathbb{E}_\xi, \mathfrak{s}_\xi)$ be a  weakly idempotent complete
 extriangulated category satisfying {\rm (Ax1)} and {\rm (Ax2)}. Assume that the pair $({^\bot}\mathcal{I}(\xi),\mathcal{I}(\xi))$ is cogenerated by a set $\mathcal{S}$ in $(\mathcal{B}, \mathbb{E}_\xi, \mathfrak{s}_\xi)$.

$(1)$ The pair $(^\perp\mathcal{I}(\xi)_m\cap\mathcal{I}(\xi)_n, \mathcal{GI}(\xi)_m)$ is a complete and hereditary cotorsion pair in the extriangulated category $(\mathcal{GI}(\xi)_n, \mathbb{E}_{\mathcal{GI}(\xi)_n}, \mathfrak{s}_{\mathcal{GI}(\xi)_n})$ with core $^\perp\mathcal{I}(\xi)_m\cap\mathcal{I}(\xi)_m$.

$(2)$ The triple $(^\perp\mathcal{I}(\xi)_m\cap\mathcal{GI}(\xi)_n, \mathcal{I}(\xi)_n, \mathcal{GI}(\xi)_m)$ is a hereditary Hovey triple in the extriangulated category $(\mathcal{GI}(\xi)_n, \mathbb{E}_{\mathcal{GI}(\xi)_n}, \mathfrak{s}_{\mathcal{GI}(\xi)_n})$, and the homotopy category is the stable category $\underline{^{\perp}\mathcal{I}(\xi)_m\cap\mathcal{GI}(\xi)_m}$, which is triangle equivalent to $\underline{\mathcal{GI}(\xi)}$.
\end{thm}

The paper is organized as follows. Section \ref{Preliminaries} recalls the necessary preliminaries on extriangulated categories, cotorsion pairs, relative homological dimension, model structures, and $\xi$-$\mathcal{G}$injective objects. Section \ref{section3} is devoted to the proof of Theorem \ref{thm2}. In Section 4, we study model structures induced by objects of finite $\xi$-$\mathcal{G}$injective dimension and prove Theorem \ref{thm4}. Section \ref{section5} constructs model structures on the extriangulated category induced by objects of $\xi$-$\mathcal{G}$injective dimension at most $n$, where $n$ is a non-negative integer, and concludes with the proof of Theorem \ref{thm4.15'}.

\section{\bf Preliminaries}\label{Preliminaries}
Throughout $\mathcal{B}$ is an additive category. All subcategories are full and closed under
isomorphisms and direct summands.

\subsection{Extriangulated categories} Let us briefly recall some definitions and basic properties of extriangulated categories from \cite{NP}.

Suppose throughout that $\mathbb{E}: \mathcal{B}^{\rm op}\times \mathcal{B}\rightarrow {\rm Ab}$ is an additive bifunctor,
where ${\rm Ab}$ is the category of abelian groups. For any objects $A, C\in\mathcal{B}$, an element $\delta\in \mathbb{E}(C,A)$ is called an \emph{$\mathbb{E}$-extension}.
Let $\mathfrak{s}$ be a correspondence which associates an equivalence class $$\mathfrak{s}(\delta)=\xymatrix@C=0.8cm{[A\ar[r]^x
 &B\ar[r]^y&C]}$$ to any $\mathbb{E}$-extension $\delta\in\mathbb{E}(C, A)$. This $\mathfrak{s}$ is called a {\it realization} of $\mathbb{E}$, if it makes the diagram in \cite[Definition 2.9]{NP} commute.
 A triplet $(\mathcal{B}, \mathbb{E}, \mathfrak{s})$ is called an {\it extriangulated category} if it satisfies the following conditions.
\begin{enumerate}
\item $\mathbb{E}\colon\mathcal{B}^{\rm op}\times \mathcal{B}\rightarrow \rm{Ab}$ is an additive bifunctor.

\item $\mathfrak{s}$ is an additive realization of $\mathbb{E}$.

\item $\mathbb{E}$ and $\mathfrak{s}$  satisfy the compatibility conditions in \cite[Definition 2.12]{NP}.

 \end{enumerate}

\begin{rem}
Note that both exact categories and triangulated categories are extriangulated categories $($see \cite[Example 2.13]{NP}$)$ and extension closed subcategories of extriangulated categories are
again extriangulated $($see \cite[Remark 2.18]{NP}$)$. Moreover, there exist extriangulated categories which
are neither exact categories nor triangulated categories $($see \cite[Proposition 3.30]{NP}, \cite[Example 4.14]{ZZ} and \cite[Remark 3.3]{HZZ}$)$.
\end{rem}

We will use the following terminology.

\begin{definition}{ \emph{(\cite[Definitions 2.15 and 2.19]{NP})}} {\rm
 Let $(\mathcal{B}, \mathbb{E}, \mathfrak{s})$ be an extriangulated category.
\begin{enumerate}
\item A sequence $\xymatrix@C=1cm{A\ar[r]^x&B\ar[r]^{y}&C}$ is called a {\it conflation} if it realizes some $\mathbb{E}$-extension $\delta\in\mathbb{E}(C, A)$.
In this case, $x$ is called an {\it inflation} and $y$ is called a {\it deflation}.

\item  If a conflation $\xymatrix@C=0.6cm{A\ar[r]^x&B\ar[r]^{y}&C}$ realizes $\delta\in\mathbb{E}(C, A)$, we call the pair
$\xymatrix@C=0.6cm{(A\ar[r]^x&B\ar[r]^{y}&C, \delta)}$ an {\it $\mathbb{E}$-triangle}, and write it in the following.
\begin{center} $\xymatrix{A\ar[r]^x&B\ar[r]^{y}&C\ar@{-->}[r]^{\delta}&}$\end{center}
We usually do not write this ``$\delta$" if it is not used in the argument.

\item Let $\xymatrix{A\ar[r]^x&B\ar[r]^{y}&C\ar@{-->}[r]^{\delta}&}$ and $\xymatrix{A'\ar[r]^{x'}&B'\ar[r]^{y'}&C'\ar@{-->}[r]^{\delta'}&}$
be any pair of $\mathbb{E}$-triangles. If a triplet $(a, b, c)$ realizes $(a, c): \delta\rightarrow \delta'$, then we write it as
 $$\xymatrix{A\ar[r]^{x}\ar[d]_{a}&B\ar[r]^{y}\ar[d]_{b}&C\ar[d]_{c}\ar@{-->}[r]^{\delta}&\\
 A'\ar[r]^{x'}&B'\ar[r]^{y'}&C'\ar@{-->}[r]^{\delta'}&}$$
 and call $(a, b, c)$ a {\it morphism} of $\mathbb{E}$-triangles.
\end{enumerate}}

\end{definition}

\begin{definition}{\rm(\cite[Proposition 3.24 and Definition 3.25]{NP})} {\rm  Let $(\mathcal{B}, \mathbb{E}, \mathfrak{s})$ be an extriangulated category.
An object $I\in\mathcal{B}$ is an {\it injective object} if $\mathbb{E}(B, I)=0$ for any object $B\in\mathcal{B}$. Denote the
subcategory of injective objects by $\mathcal{I}\subseteq\mathcal{B}$. One says that $\mathcal{B}$ has {\it enough injectives} if
any object $B\in\mathcal{B}$ admits an inflation $B\longrightarrow I$ for some $I\in\mathcal{I}$.

Dually, one has the notions of {\it projective objects} and {\it having enough projectives.}}
  \end{definition}

  Liu and Nakaoka \cite{LN} defined the higher extension groups in an extriangulated
category with enough projectives and enough injectives. They showed the following result.

\begin{lem} $($\cite[Proposition 5.2]{LN}$)$\label{lem2.11} Let $(\mathcal{B}, \mathbb{E}, \mathfrak{s})$ be an extriangulated category with enough projectives and enough injectives, and $\xymatrix@C=2em{A\ar[r]^f&B\ar[r]^g&C\ar@{-->}[r]^\delta&}$
be an $\mathbb{E}$-triangle. There are long exact sequences

$\xymatrix@=2em{\cdots\ar[r]&\mathbb{E}^i(X, A)\ar[r]^{f_*}&\mathbb{E}^i(X, B)\ar[r]^{g_*}&\mathbb{E}^i(X, C)\ar[r]&\mathbb{E}^{i+1}(X, A)\ar[r]^{f_*}&\mathbb{E}^{i+1}(X, B)\ar[r]&\cdots}$

$\xymatrix@=2em{\cdots\ar[r]&\mathbb{E}^i(C, X)\ar[r]^{g^*}&\mathbb{E}^i(B, X)\ar[r]^{f^*}&\mathbb{E}^i(A, X)\ar[r]&\mathbb{E}^{i+1}(C, X)\ar[r]^{g^*}&\mathbb{E}^{i+1}(B, X)\ar[r]&\cdots}$\\
for any
object $X\in\mathcal{B}$ and $i\geqslant 0$.\end{lem}

The following condition is analogous to the weakly idempotent completeness in exact category (see \cite[Condition 5.8]{NP}).

\begin{cond} \label{cond:4.11} \emph{({\rm Condition (WIC)})}  Consider the following conditions.

\begin{enumerate}
\item[\rm (1)]  Let $f\in\mathcal{B}(A, B), g\in\mathcal{B}(B, C)$ be any composable pair of morphisms. If $gf$ is an inflation, then so is $f$.

\item[\rm (2)] Let $f\in\mathcal{B}(A, B), g\in\mathcal{B}(B, C)$ be any composable pair of morphisms. If $gf$ is a deflation, then so is $g$.
\end{enumerate}

\end{cond}

\begin{remark}\label{Re:4.12}

\emph{(1)} If $\mathcal{B}$ is an exact category, then Condition \emph{(WIC)} is equivalent to $\mathcal{B}$ is
weakly idempotent complete \emph{(see \cite[Proposition 7.6]{B"u})}.

\emph{(2)} If $\mathcal{B}$ is a triangulated category, then Condition \emph{(WIC)} is automatically satisfied.
\end{remark}

It is worth noting that an extriangulated category $\mathcal{B}$ satisfies Condition {\rm (WIC)} if and only if $\mathcal{B}$ is weakly idempotent complete (see \cite[Proposition C]{K}).

{\bf From now until the end of the paper, we always assume that $\mathcal{B}:=(\mathcal{B}, \mathbb{E}, \mathfrak{s})$ is a WIC extriangulated category.}

\begin{lem}\label{lem1} \emph{(\cite[Proposition 3.15]{NP})}
Let $\xymatrix@C=1.8em{A_1\ar[r]^{x_1}&B_1\ar[r]^{y_1}&C\ar@{-->}[r]^{\delta_1}&}$ and $\xymatrix@C=1.8em{A_2\ar[r]^{x_2}&B_2\ar[r]^{y_2}&C\ar@{-->}[r]^{\delta_2}&}$ be any pair of $\mathbb{E}$-triangles. Then there exists a commutative diagram
in $\mathcal{B}$
$$\xymatrix{
    & A_2\ar[d]_{m_2} \ar@{=}[r] & A_2 \ar[d]^{x_2} \\
  A_1 \ar@{=}[d] \ar[r]^{m_1} & M \ar[d]_{e_2} \ar[r]^{e_1} & B_2\ar[d]^{y_2} \\
  A_1 \ar[r]^{x_1} & B_1\ar[r]^{y_1} & C   }
  $$
  which satisfies $\mathfrak{s}(y^*_2\delta_1)=\xymatrix@C=2em{[A_1\ar[r]^{m_1}&M\ar[r]^{e_1}&B_2]}$ and
  $\mathfrak{s}(y^*_1\delta_2)=\xymatrix@C=2em{[A_2\ar[r]^{m_2}&M\ar[r]^{e_2}&B_1].}$

\end{lem}

The following definitions are quoted verbatim from \cite[Section 3]{HZZ}. A class of $\mathbb{E}$-triangles $\xi$ is {\it closed under base change} if for any $\mathbb{E}$-triangle $$\xymatrix@C=2em{A\ar[r]^x&B\ar[r]^y&C\ar@{-->}[r]^{\delta}&\in\xi}$$ and any morphism $c\colon C' \to C$, any $\mathbb{E}$-triangle  $\xymatrix@C=2em{A\ar[r]^{x'}&B'\ar[r]^{y'}&C'\ar@{-->}[r]^{c^*\delta}&}$ belongs to $\xi$.

Dually, a class of  $\mathbb{E}$-triangles $\xi$ is {\it closed under cobase change} if for any $\mathbb{E}$-triangle $$\xymatrix@C=2em{A\ar[r]^x&B\ar[r]^y&C\ar@{-->}[r]^{\delta}&\in\xi}$$ and any morphism $a\colon A \to A'$, any $\mathbb{E}$-triangle  $\xymatrix@C=2em{A'\ar[r]^{x'}&B'\ar[r]^{y'}&C\ar@{-->}[r]^{a_*\delta}&}$ belongs to $\xi$.

A class of $\mathbb{E}$-triangles $\xi$ is called {\it saturated} if in the situation of Lemma \ref{lem1}, whenever  \\
$\xymatrix@C=2em{A_2\ar[r]^{x_2}&B_2\ar[r]^{y_2}&C\ar@{-->}[r]^{\delta_2 }&}$
 and $\xymatrix@C=2em{A_1\ar[r]^{m_1}&M\ar[r]^{e_1}&B_2\ar@{-->}[r]^{y_2^{\ast}\delta_1}&}$
 belong to $\xi$,  the  $\mathbb{E}$-triangle $$\xymatrix@C=2em{A_1\ar[r]^{x_1}&B_1\ar[r]^{y_1}&C\ar@{-->}[r]^{\delta_1 }&}$$  belongs to $\xi$.

An $\mathbb{E}$-triangle $\xymatrix@C=2em{A\ar[r]^x&B\ar[r]^y&C\ar@{-->}[r]^{\delta}&}$ is called {\it split} if $\delta=0$. It is easy to see that it is split if and only if $x$ is section or $y$ is retraction. The full subcategory  consisting of the split $\mathbb{E}$-triangles will be denoted by $\Delta_0$.

  \begin{definition} \emph{(\cite[Definition 3.1]{HZZ})}\label{def:proper class} {\rm  Let $\xi$ be a class of $\mathbb{E}$-triangles which is closed under isomorphisms. Then $\xi$ is called a {\it proper class} of $\mathbb{E}$-triangles if the following conditions hold:

  \begin{enumerate}
\item  $\xi$ is closed under finite coproducts and $\Delta_0\subseteq \xi$.

\item $\xi$ is closed under base change and cobase change.

\item $\xi$ is saturated.

  \end{enumerate}}
  \end{definition}

 \begin{definition} \emph{(\cite[Definition 4.1]{HZZ})}
 {\rm An object $I\in\mathcal{B}$  is called {\it $\xi$-injective}  if for any $\mathbb{E}$-triangle $$\xymatrix{A\ar[r]^x& B\ar[r]^y& C \ar@{-->}[r]^{\delta}& }$$ in $\xi$, the induced sequence of abelian groups $\xymatrix@C=0.6cm{0\ar[r]& \mathcal{B}(C,I)\ar[r]& \mathcal{B}(B,I)\ar[r]&\mathcal{B}(A,I)\ar[r]& 0}$ is exact. Dually, we have the definition of {\it $\xi$-projective}.}
\end{definition}

We denote by $\mathcal{I(\xi)}$ (resp., $\mathcal{P(\xi)}$) the class of $\xi$-injective (resp., $\xi$-projective) objects of $\mathcal{B}$. It follows from the definition that this subcategory $\mathcal{I}(\xi)$ and $\mathcal{P}(\xi)$ are full, additive, closed under isomorphisms and direct summands.

 An extriangulated  category $(\mathcal{B}, \mathbb{E}, \mathfrak{s})$ is said to  have {\it  enough
$\xi$-injectives} \ (resp., {\it  enough $\xi$-projectives}) provided that for each object $A$ there exists an $\mathbb{E}$-triangle $\xymatrix@C=2em{A\ar[r]& I\ar[r]& K\ar@{-->}[r]&}$ (resp., $\xymatrix@C=2.1em{K\ar[r]& P\ar[r]&A\ar@{-->}[r]& }$) in $\xi$ with $I\in\mathcal{I}(\xi)$ (resp., $P\in\mathcal{P}(\xi)$).\\

\subsection{Cotorsion pairs} 
Let $\mathcal{X}$ and $\mathcal{Y}$ be subcategories of $\mathcal{B}$.  If $\mathbb{E}(X, Y)=0$ for all $X\in\mathcal{X}$ and $Y\in\mathcal{Y}$, then we write $\mathbb{E}(\mathcal{X}, \mathcal{Y})=0$. Put
\vspace{2mm}
\begin{center}
$\mathcal{X}^\perp=\{Z\in\mathcal{B}\mid \mathbb{E}(X, Z)=0, \forall X\in\mathcal{X}\}$\hspace{3mm}and\hspace{3mm}$^\perp\mathcal{Y}=\{Z\in\mathcal{B}\mid \mathbb{E}(Z, Y)=0, \forall Y\in\mathcal{Y}\}$.
\end{center}
Given an $\mathbb{E}$-triangle $\xymatrix{A \ar[r]^{x} & B \ar[r]^{y} & C \ar@{-->}[r]^{\delta}&,}$ we call $A$ the \emph{cocone} of $y:B\rightarrow C$ and $C$ the \emph{cone} of $x:A\rightarrow B$.

\begin{definition} {\rm Let 
$\mathcal{Z}$ be a subcategory of $\mathcal{B}$.

(1) $\mathcal{Z}$ is {\it closed under cocones of deflations} if for any $\mathbb{E}$-triangle $\xymatrix{A\ar[r]^x&B\ar[r]^y&C\ar@{-->}[r]^\delta&}$ which satisfies $B, C\in\mathcal{Z}$, we have $A\in\mathcal{Z}$.

(2) $\mathcal{Z}$ is {\it closed under cones of inflations} if for any $\mathbb{E}$-triangle $\xymatrix{A\ar[r]^x&B\ar[r]^y&C\ar@{-->}[r]^\delta&}$ which satisfies $A, B\in\mathcal{Z}$, we have $C\in\mathcal{Z}$.

(3) $\mathcal{Z}$ is {\it closed under extensions} if for any $\mathbb{E}$-triangle $\xymatrix{A\ar[r]^x&B\ar[r]^y&C\ar@{-->}[r]^\delta&}$ which satisfies $A, C\in\mathcal{Z}$, we have $B\in\mathcal{Z}$.}
\end{definition}

\begin{definition} {\rm 

 (1) A pair $(\mathcal{X, Y})$  of classes of objects in $\mathcal{B}$ is  {\it cotorsion pair}, provided that $\mathcal{X}^\perp=\mathcal{Y}$ and $^\perp\mathcal{Y}=\mathcal{X}$.

 (2) A  cotorsion pair $(\mathcal{X}, \mathcal{Y})$ is {\it complete} if for any object $C\in\mathcal{B}$, there are $\mathbb{E}$-triangles
 \vspace{2mm}
\begin{center}
  $\xymatrix{Y_C\ar[r]&X_C\ar[r]&C\ar@{-->}[r]&}$ and $\xymatrix{C\ar[r]&Y^C\ar[r]&X^C\ar@{-->}[r]&,}$
\end{center}
where $X_C, X^C\in\mathcal{X}$ and  $Y_C, Y^C\in\mathcal{Y}$. In this case,
we call $\mathcal X\cap\mathcal Y$ the core of the cotorsion pair $(\mathcal X,\mathcal Y)$.

(3) A cotorsion pair $(\mathcal{X}, \mathcal{Y})$ is \emph{hereditary} if $\mathcal{X}$ is closed under cocones of deflations and $\mathcal{Y}$ is closed under cones of inflations.
\vspace{1mm}}
\end{definition}

Sometimes it is convenient to use the following description of complete cotorsion pair.

\begin{fact}\label{fact2.18} {\rm Let 
$\mathcal{X, Y}$ be two classes of objects in $\mathcal{B}$ which are closed under isomorphisms and direct summands. Then $(\mathcal{X}, \mathcal{Y})$ is a complete cotorsion pair if it satisfies the following conditions:

(a) $\mathbb{E}(\mathcal{X}, \mathcal{Y})=0$.

(b) For any $C\in\mathcal{B}$, there are $\mathbb{E}$-triangles

\qquad $\xymatrix{Y_C\ar[r]&X_C\ar[r]&C\ar@{-->}[r]&}$ and $\xymatrix{C\ar[r]&Y^C\ar[r]&X^C\ar@{-->}[r]&}$

\noindent satisfying $X_C, X^C\in\mathcal{X}$ and  $Y_C, Y^C\in\mathcal{Y}$.
\vspace{1mm}}
\end{fact}

Complete cotorsion pair here is also called cotorsion pair in \cite{NP, HZZ, AT}. The following result gives several equivalent conditions of a cotorsion pair being complete.

 \begin{lem}\label{lemma2.18} Let $(\mathcal{B}, \mathbb{E},\mathfrak{s})$ be an extriangulated category with enough projectives and enough injectives, and $(\mathcal{X}, \mathcal{Y})$ a cotorsion pair in $(\mathcal{B}, \mathbb{E},\mathfrak{s})$. Then the following statements are equivalent:

{\rm (1)}~ $(\mathcal{X}, \mathcal{Y})$ is complete.

{\rm (2)}~For any $A\in \mathcal{B}$, there exists an $\mathbb{E}$-triangle $\xymatrix{Y'\ar[r]&X'\ar[r]&A\ar@{-->}[r]&}$ with $X'\in \mathcal{X}$ and $Y'\in\mathcal{Y}$.

{\rm (3)}~For any $A\in \mathcal{B}$, there exists an $\mathbb{E}$-triangle $\xymatrix@=2em{A\ar[r]&Y\ar[r]&X\ar@{-->}[r]&}$ with $X\in \mathcal{X}$ and $Y\in\mathcal{Y}$.
\end{lem}
 \begin{proof} $(1)\Rightarrow(2)$ and $(1)\Rightarrow (3)$ are obvious.

 $(2)\Rightarrow (3).$ For any $A\in \mathcal{B}$, there is an $\mathbb{E}$-triangle $\xymatrix@=2em{A\ar[r]&I\ar[r]&K\ar@{-->}[r]&}$ with $I\in \mathcal{I}$ since $(\mathcal{B}, \mathbb{E},\mathfrak{s})$ has enough injectives. By (2), there exists an $\mathbb{E}$-triangle $\xymatrix{Y'\ar[r]&X'\ar[r]&K\ar@{-->}[r]&}$ with $X'\in \mathcal{X}$ and $Y'\in\mathcal{Y}$. It follows from Lemma \ref{lem1} that there is a commutative diagram of $\mathbb{E}$-triangles
 $$\xymatrix{&Y'\ar[d]\ar@{=}[r]&Y'\ar[d]&\\
A\ar@{=}[d]\ar[r]&Y\ar[r]\ar[d]&X'\ar@{-->}[r]\ar[d]&\\
A\ar[r]&I\ar@{-->}[d]\ar[r]&K\ar@{-->}[d]\ar@{-->}[r]&\\
&&&}$$
 Since $Y'\in\mathcal{Y}$ and $I\in\mathcal{I}\subseteq\mathcal{Y}$, one gets $Y\in\mathcal{Y}$. Hence there is  $\mathbb{E}$-triangle $\xymatrix@=2em{A\ar[r]&Y\ar[r]&X'\ar@{-->}[r]&}$ with $X'\in \mathcal{X}$ and $Y\in\mathcal{Y}$.

 $(3)\Rightarrow (2).$ The proof is dual to that of $(2)\Rightarrow (3)$.

 $(2)\Rightarrow (1)$ is clear because the equivalence of (2) and (3).
 \end{proof}

 The following result gives several equivalent conditions of a complete cotorsion pair being hereditary.

\begin{lem}\label{lem2.18}$($\cite[Proposition 2.18]{HZZZ}$)$ Let 
$(\mathcal{X}, \mathcal{Y})$ be a complete cotorsion pair. Then the following statements are equivalent:

{\rm (1)} \  $\mathcal{Y}$ is closed under cones of inflations.

{\rm (2)} \  $\mathcal{X}$ is closed under cocones of deflations.

{\rm (3)} \   $(\mathcal{X}, \mathcal{Y})$ is hereditary.\\
Moreover, if $\mathcal{B}$ has enough projectives and enough
 injectives, then the above conditions are also equivalent to:

{\rm (4)} \   $\mathbb{E}^{i}(X, Y)=0$ for each $i\geqslant2$ and all $X\in{\mathcal{X}}$ and $Y\in{\mathcal{Y}}$.

\end{lem}

\subsection{Relative homological dimension} Recall that an {\it$\mathbb{E}$-triangle sequence} is a pair $(\mathbf{X}, Z_\bullet(\mathbf{X}))$ where $\mathbf{X}$ is a sequence $$\mathbf{X}:=\quad\xymatrix@C=2em{\cdots\ar[r]&X_{n+1}\ar[r]^{d_{n+1}}&X_n\ar[r]^{d_n}&X_{n-1}\ar[r]&\cdots}$$ in $\mathcal{B}$, and $Z_\bullet(\mathbf{X})$ is a family of $\mathbb{E}$-triangles   $\{\xymatrix@C=2em{Z_{n+1}(\mathbf{X})\ar[r]^{~~~g_n}&X_n\ar[r]^{f_n~~~}&Z_{n}(\mathbf{X})\ar@{-->}[r]^{\delta_n}&}\}_{n\in\mathbb{Z}}$ satisfying that $d_n=g_{n-1}f_n$. Notice that $\mathbf{X}$ is a chain of complex and $Z_n(\mathbf{X})$ is called the $n$th $\mathbb{E}$-cycle of $\mathbf{X}$ in $\mathcal{B}$. For the sake of simplicity, an $\mathbb{E}$-triangle sequence $(\mathbf{X}, Z_\bullet(\mathbf{X}))$ will be denote by $\mathbf{X}$.

Let $\mathcal{Y}$ be a subcategory of an extriangulated category $(\mathcal{B}, \mathbb{E}, \mathfrak{s})$ and $B\in\mathcal{B}$. A {\it$\mathcal{Y}$-coresolution} of $B$ is an $\mathbb{E}$-triangle sequence of the form
 $$\mathbf{Y}:=\quad \xymatrix@C=2em{0\ar[r]&B\ar[r]&Y_0\ar[r]&Y_{-1}\ar[r]&\cdots\ar[r]&Y_{-k}\ar[r]&\cdots}$$
 where $Y_{-k}\in\mathcal{Y}$ for every $k\geqslant 0$. If $Y_{-k}=0$ for every $k>n$, we shall say the previous sequence is a {\it finite $\mathcal{Y}$-coresolution of $B$ of length $n$}. The {\it $\mathcal{Y}$-coresolution dimension of $B$}, denoted by $\mathcal{Y}$-${\rm coresdim}(B)$, is the smallest $n\geqslant 0$ such that $B$ admits a finite $\mathcal{Y}$-coresolution of length $n$. If such $n$ does not exist, we set $\mathcal{Y}$-${\rm coresdim}(B):=\infty$. We denote by $\mathcal{Y}_n$ the full subcategory of $\mathcal{B}$ consisting of all objects of $\mathcal{Y}$-coresolution dimension at most $n$.

Let $(\mathcal{B}, \mathbb{E}, \mathfrak{s})$ be an extriangulated category. A subcategory $\mathcal{Y}$ of $\mathcal{B}$ is called {\it special preenveloping} if for any object $B\in\mathcal{B}$, there exists an $\mathbb{E}$-triangle $\xymatrix{B\ar[r]&Y\ar[r]&K\ar@{-->}[r]&}$ with $Y\in\mathcal{Y}$ and $K\in{^\perp\mathcal{Y}}$. {\it Special precovering subcetegory} can be defined dually. Note that a cotorsion pair $(^\perp\mathcal{Y}, \mathcal{Y})$ is complete if and only if $^\perp\mathcal{Y}$ is a special precovering and $\mathcal{Y}$ is a special preenveloping.

A $\mathcal{Y}$-coresolution of $B$

$$\mathbf{Y}:=\quad\xymatrix@C=2em{0\ar[r]&B\ar[r]&Y_0\ar[r]&Y_{-1}\ar[r]&\cdots\ar[r]&Y_{-k}\ar[r]&\cdots}$$
is called {\it special proper $\mathcal{Y}$-coresolution of $B$} if $Z_{-k}(\mathbf{Y})\in{^\perp\mathcal{Y}}$ for every $k\geqslant 0$.

\begin{lem}\label{lem2.19} Let $(^\perp\mathcal{Y}, \mathcal{Y})$ be a complete and hereditary cotorsion pair in an extriangulated category $(\mathcal{B}, \mathbb{E}, \mathfrak{s})$. Then the following are equivalent for any object $B\in\mathcal{B}$ and any integer $m\geqslant 1$:

$(1)$ $B\in\mathcal{Y}_m$;

$(2)$ For any special proper $\mathcal{Y}$-coresolution

$$\mathbf{Y}:=\quad\xymatrix@C=2em{0\ar[r]&B\ar[r]&Y_0\ar[r]&Y_{-1}\ar[r]&\cdots\ar[r]&Y_{-k}\ar[r]&\cdots},$$
we have $Z_{-m+1}(\mathbf{Y})\in\mathcal{Y}$;

$(3)$ There exists a special proper $\mathcal{Y}$-coresolution

$$\mathbf{Y}:=\quad\xymatrix@C=2em{0\ar[r]&B\ar[r]&Y_0\ar[r]&Y_{-1}\ar[r]&\cdots\ar[r]&Y_{-m+1}\ar[r]&Y_{-m}\ar[r]&0}.$$

$(4)$ For any $\mathcal{Y}$-coresolution

$$\mathbf{Y'}:=\quad\xymatrix@C=2em{0\ar[r]&B\ar[r]&Y_0'\ar[r]&Y_{-1}'\ar[r]&\cdots\ar[r]&Y_{-k}'\ar[r]&\cdots},$$
we have $Z_{-m+1}(\mathbf{Y'})\in\mathcal{Y}$.
\end{lem}

\begin{proof} $(1)\Rightarrow (2).$ We proceed by induction.

If $m=1$, then $B\in\mathcal{Y}$ and there exists an $\mathbb{E}$-triangle $\xymatrix@C=2em{0\ar[r]&B\ar[r]&X\ar[r]&Z\ar@{-->}[r]&}$ with $X, Z\in\mathcal{Y}$. Let us show that $Z_0(\mathbf{Y})\in\mathcal{Y}$. It follows from dual version of Lemma \ref{lem1} that there is a commutative diagram with $\mathbb{E}$-triangles
 $$\xymatrix{B\ar[r]\ar[d]&X\ar[r]\ar[d]&Z\ar@{-->}[r]\ar@{=}[d]&\\
Y_0\ar[d]\ar[r]&K\ar[r]\ar[d]&Z\ar@{-->}[r]&\\
Z_0(\mathbf{Y})\ar@{=}[r]\ar@{-->}[d]&Z_0(\mathbf{Y})\ar@{-->}[d]&&\\
&&&}$$
Hence $K\in\mathcal{Y}$ as $Y_0, Z\in\mathcal{Y}$ and $\mathcal{Y}$ is closed under extensions. Since $Z_0(\mathbf{Y})\in{^\perp\mathcal{Y}}$ and $X\in\mathcal{Y}$, the $\mathbb{E}$-triangle $\xymatrix{X\ar[r]&K\ar[r]&Z_0(\mathbf{Y})\ar@{-->}[r]&}$ splits, and $Z_0(\mathbf{Y})\in\mathcal{Y}$.

For $m\geqslant 2$, there exists an $\mathbb{E}$-triangle $\xymatrix@C=2em{B\ar[r]&X\ar[r]&Z\ar@{-->}[r]&}$ with $X\in\mathcal{Y}$ and $Z\in\mathcal{Y}_{m-1}$. Since $Z_0(\mathbf{Y})\in{^\perp\mathcal{Y}}$, we can construct the following commutative diagram with $\mathbb{E}$-triangles

$$\xymatrix{B\ar[r]\ar[d]&Y_0\ar[r]\ar[d]&Z_0(\mathbf{Y})\ar[d]\ar@{-->}[r]&\\X\ar[r]\ar[d]&X\oplus Y_{-1}\ar[r]\ar[d]&Y_{-1}\ar[d]\ar@{-->}[r]&\\Z\ar[r]\ar@{-->}[d]&L\ar[r]\ar@{-->}[d]&Z_{-1}(\mathbf{Y})\ar@{-->}[r]\ar@{-->}[d]&\\
&&&}$$

\noindent by \cite[Lemm 5]{HZZ2}. Hence $L\in\mathcal{Y}$ as $(^\perp\mathcal{Y}, \mathcal{Y})$ is a hereditary cotorsion pair and $Y_0, X, Y_{-1}\in\mathcal{Y}$. Now $Z\in\mathcal{Y}_{m-1}$ and $Z$ has the following special proper $\mathcal{Y}$-coresolution

$$\mathbf{Y'}:=\quad\xymatrix@C=2em{0\ar[r]&Z\ar[r]&Y_0'\ar[r]&Y_{-1}'\ar[r]&\cdots\ar[r]&Y_{-k}'\ar[r]&\cdots}$$

\noindent with $Y_0'=L$ and $Y'_{-k}=Y_{-k-1}$ for any $k\geqslant 1$. By induction $Z_{-m+2}(\mathbf{Y'})=Z_{-m+1}(\mathbf{Y})\in\mathcal{Y}$, as desired.

$(2)\Rightarrow (3).$ By hypothesis and the fact that $\mathcal{Y}$ is a special preenveloping class.

$(3)\Rightarrow (4).$ Assume that there exists a special proper $\mathcal{Y}$-coresolution of $B$

$$\mathbf{Y}:=\quad\xymatrix@C=2em{0\ar[r]&B\ar[r]&Y_0\ar[r]&Y_{-1}\ar[r]&\cdots\ar[r]&Y_{-m+1}\ar[r]&Y_{-m}\ar[r]&0}.$$

For any  $\mathcal{Y}$-coresolution

$$\mathbf{Y'}:=\quad\xymatrix@C=2em{0\ar[r]&B\ar[r]&Y_0'\ar[r]&Y_{-1}'\ar[r]&\cdots\ar[r]&Y_{-k}'\ar[r]&\cdots}.$$

\noindent Since $Z_{0}(\mathbf{Y})\in{^\perp\mathcal{Y}}$, it follows from \cite[Lemm 5]{HZZ2} that there exists a commutative diagram with $\mathbb{E}$-triangles
$$\xymatrix{B\ar[r]\ar[d]&Y_0\ar[r]\ar[d]&Z_0(\mathbf{Y})\ar[d]\ar@{-->}[r]&\\
Y'_0\ar[r]\ar[d]&Y'_0\oplus Y_{-1}\ar[r]\ar[d]&Y_{-1} \ar[d]\ar@{-->}[r]&\\
Z_0(\mathbf{Y'})\ar[r]\ar@{-->}[d]&K_{0}\ar[r]\ar@{-->}[d]&Z_{-1}(\mathbf{Y})\ar@{-->}[r]\ar@{-->}[d]&\\
&&&}$$

\noindent Since $Z_{-k}(\mathbf{Y})\in{^\perp\mathcal{Y}}$ for $k=0, 1, \cdots, m-1$ and $Z_{-m+1}(\mathbf{Y})=Y_{-m}$, we can obtain the following $\mathbb{E}$-triangle sequence

$$\xymatrix{0\ar[r]&Y_0\ar[r]&Y'_0\oplus Y_{-1}\ar[r]&\cdots\ar[r]&Y'_{-m+2}\oplus Y_{-m+1}\ar[r]&K_{-m+2}\ar[r]&0}$$

\noindent and the $\mathbb{E}$-triangle $\xymatrix{Z_{-m+2}(\mathbf{Y'})\ar[r]&K_{-m+2}\ar[r]&Y_{-m}\ar@{-->}[r]&}$ in $\mathcal{B}$
by repeatedly applying \cite[Lemm 5]{HZZ2}. It is easy to see that $K_{-m+2}\in\mathcal{Y}$ as $Y_0, Y_{-1}, \cdots, Y_{-m+1}$, $Y'_0, Y'_{-1}, \cdots, Y'_{-m+2}\in\mathcal{Y}$ and $\mathcal{Y}$ is closed under cones of inflations. It follows from dual version of Lemma \ref{lem1} that there is a commutative diagram with $\mathbb{E}$-triangles
 $$\xymatrix{Z_{-m+2}(\mathbf{Y'})\ar[r]\ar[d]&K_{-m+2}\ar[r]\ar[d]&Y_{-m}\ar@{-->}[r]\ar@{=}[d]&\\
Y'_{-m+1}\ar[d]\ar[r]&L\ar[r]\ar[d]&Y_{-m}\ar@{-->}[r]&\\
Z_{-m+1}(\mathbf{Y'})\ar@{=}[r]\ar@{-->}[d]&Z_{-m+1}(\mathbf{Y'})\ar@{-->}[d]&&\\
&&&}$$
Hence $L\in\mathcal{Y}$ as $Y'_{-m+1}, Y_{-m}\in\mathcal{Y}$ and $\mathcal{Y}$ is closed under extensions. Therefore,  $Z_{-m+1}(\mathbf{Y'})\in\mathcal{Y}$ since $K_{-m+2}, L\in\mathcal{Y}$ and $\mathcal{Y}$ is closed under cones of inflations.

$(4)\Rightarrow (1)$ is obvious.
\end{proof}

\begin{prop}\label{cor2.20} Let $(^\perp\mathcal{Y}, \mathcal{Y})$ be a complete and hereditary cotorsion pair in an extriangulated category $(\mathcal{B}, \mathbb{E}, \mathfrak{s})$, and $n$ a non-negative integer.
Then for any $\mathbb{E}$-triangle $\xymatrix@C=2em{A\ar[r]&B\ar[r]&C\ar@{-->}[r]&}$ in $\mathcal{B}$, the following hold:

$(1)$ If $A, C\in\mathcal{Y}_n$, then $B\in\mathcal{Y}_n$.

$(2)$ If $B\in\mathcal{Y}_n$ and $A\in\mathcal{Y}_{n+1}$, then $C\in\mathcal{Y}_n$.

$(3)$ If $B\in\mathcal{Y}_{n+1}$ and $C\in\mathcal{Y}_{n}$, then $A\in\mathcal{Y}_{n+1}$.

\end{prop}
\begin{proof} We proceed by induction.

(1) Since $\mathcal{Y}$ is closed under extensions, the result holds for the case $n=0$.

Let $n\geqslant 1$. If $A, C\in\mathcal{Y}_{n}$, then there is an $\mathbb{E}$-triangle $\xymatrix{A\ar[r]&Y\ar[r]&Z\ar@{-->}[r]&}$ with $Y\in\mathcal{Y}$ and $Z\in\mathcal{Y}_{n-1}$. Hence there is a commutative diagram with $\mathbb{E}$-triangles
 $$\xymatrix{A\ar[r]\ar[d]&B\ar[r]\ar[d]&C\ar@{-->}[r]\ar@{=}[d]&\\
Y\ar[d]\ar[r]&D\ar[r]\ar[d]&C\ar@{-->}[r]&\\
Z\ar@{=}[r]\ar@{-->}[d]&Z\ar@{-->}[d]&&\\
&&&}$$
by the dual version of Lemma \ref{lem1}. Since cotorsion pair $(^\perp\mathcal{Y}, \mathcal{Y})$ is complete, there is an $\mathbb{E}$-triangle
$\xymatrix{D\ar[r]&Y_1\ar[r]&X\ar@{-->}[r]&}$ with $Y_1\in\mathcal{Y}$ and $X\in{^\perp\mathcal{Y}}$. It follows from (ET4) that there is a commutative diagram with $\mathbb{E}$-triangles
 $$\xymatrix{Y\ar[r]\ar@{=}[d]&D\ar[r]\ar[d]&C\ar@{-->}[r]\ar[d]&\\
Y\ar[r]&Y_1\ar[r]\ar[d]&H\ar@{-->}[r]\ar[d]&\\
&X\ar@{=}[r]\ar@{-->}[d]&X\ar@{-->}[d]&\\
&&&}$$
Hence $H\in\mathcal{Y}$ as $Y, Y_1\in\mathcal{Y}$ and $\mathcal{Y}$ is closed under cones of inflations, which implies that $X\in\mathcal{Y}_{n-1}$ by Lemma \ref{lem2.19} because
$C\in\mathcal{Y}_n$ and $H\in\mathcal{Y}$. It follows from (ET4) that there is a commutative diagram with $\mathbb{E}$-triangles
 $$\xymatrix{B\ar[r]\ar@{=}[d]&D\ar[r]\ar[d]&Z\ar@{-->}[r]\ar[d]&\\
B\ar[r]&Y_1\ar[r]\ar[d]&L\ar@{-->}[r]\ar[d]&\\
&X\ar@{=}[r]\ar@{-->}[d]&X\ar@{-->}[d]&\\
&&&}$$
 Hence $L\in\mathcal{Y}_{n-1}$ as $Z, X\in\mathcal{Y}_{n-1}$ and $\mathcal{Y}_{n-1}$ is closed under extensions by induction, which implies that $B\in\mathcal{Y}_{n}$ since $Y_1\in\mathcal{Y}$ and $L\in\mathcal{Y}_{n-1}$.

(2) It follows from Lemma \ref{lem2.19} that the result holds for the case $n=0$.

Let $n\geqslant 1$. If $B\in\mathcal{Y}_n$ and $A\in\mathcal{Y}_{n+1}$, then there is an $\mathbb{E}$-triangle $\xymatrix@C=0.8cm{A\ar[r]&Y\ar[r]&Z\ar@{-->}[r]&}$ with $Y\in\mathcal{Y}$ and $Z\in\mathcal{Y}_{n}$. Hence there is a commutative diagram with $\mathbb{E}$-triangles
 $$\xymatrix{A\ar[r]\ar[d]&B\ar[r]\ar[d]&C\ar@{-->}[r]\ar@{=}[d]&\\
Y\ar[d]\ar[r]&D\ar[r]\ar[d]&C\ar@{-->}[r]&\\
Z\ar@{=}[r]\ar@{-->}[d]&Z\ar@{-->}[d]&&\\
&&&}$$
by the dual version of Lemma \ref{lem1}. Then $D\in\mathcal{Y}_n$ as $B, Z\in\mathcal{Y}_n$ and $\mathcal{Y}_n$ is closed under extensions by (1). So there is an $\mathbb{E}$-triangle $\xymatrix{D\ar[r]&Y_1\ar[r]&X\ar@{-->}[r]&}$ with $Y_1\in\mathcal{Y}$ and $X\in\mathcal{Y}_{n-1}$. It follows from (ET4) that there is a commutative diagram with $\mathbb{E}$-triangles

$$\xymatrix{Y\ar[r]\ar@{=}[d]&D\ar[r]\ar[d]&C\ar@{-->}[r]\ar[d]&\\
Y\ar[r]&Y_1\ar[r]\ar[d]&L\ar@{-->}[r]\ar[d]&\\
&X\ar@{=}[r]\ar@{-->}[d]&X\ar@{-->}[d]&\\
&&&}$$
Hence $L\in\mathcal{Y}$ as $Y, Y_1\in\mathcal{Y}$ and $\mathcal{Y}$ is closed under cones of inflations, which implies that $C\in\mathcal{Y}_n$ since $L\in\mathcal{Y}$ and $X\in\mathcal{Y}_{n-1}$.

(3) Let $n=0$. If $B\in\mathcal{Y}_1$ and $C\in\mathcal{Y}$, then there is an $\mathbb{E}$-triangle $\xymatrix{B\ar[r]&K\ar[r]&L\ar@{-->}[r]&}$ with $K, L\in\mathcal{Y}$. It follows from (ET4) that there is a commutative diagram with $\mathbb{E}$-triangles
  $$\xymatrix{A\ar[r]\ar@{=}[d]&B\ar[r]\ar[d]&C\ar@{-->}[r]\ar[d]&\\
A\ar[r]&K\ar[r]\ar[d]&Y\ar@{-->}[r]\ar[d]&\\
&L\ar@{=}[r]\ar@{-->}[d]&L\ar@{-->}[d]&\\
&&&}$$
Hence $Y\in\mathcal{Y}$ as $C, L\in\mathcal{Y}$ and $\mathcal{Y}$ is closed under extensions, which implies that $A\in\mathcal{Y}_1$.

Let $n\geqslant 1$. If $B\in\mathcal{Y}_{n+1}$ and $C\in\mathcal{Y}_{n}$, then there is an $\mathbb{E}$-triangle $\xymatrix{B\ar[r]&Y\ar[r]&Z\ar@{-->}[r]&}$ with $Y\in\mathcal{Y}$ and $Z\in\mathcal{Y}_{n}$. It follows from (ET4) that there is a commutative diagram with $\mathbb{E}$-triangles

$$\xymatrix{A\ar[r]\ar@{=}[d]&B\ar[r]\ar[d]&C\ar@{-->}[r]\ar[d]&\\
A\ar[r]&Y\ar[r]\ar[d]&X\ar@{-->}[r]\ar[d]&\\
&Z\ar@{=}[r]\ar@{-->}[d]&Z\ar@{-->}[d]&\\
&&&}$$
Since $C, Z\in\mathcal{Y}_n$ and $\mathcal{Y}_n$ is closed under extensions by (1), we have $X\in\mathcal{Y}_n$, which implies that $A\in\mathcal{Y}_{n+1}$ as $Y\in\mathcal{Y}$ and $X\in\mathcal{Y}_n$.
\end{proof}

\subsection{Model structures} We refer to \cite{Q1,HC,G} for some basic concepts and results on
model structures.

Recall from \cite[Definition 5.5]{NP} that an \emph{admissible model structure} on a WIC extriangulated category $\mathcal{B}$ is a model structure in the sense of \cite[Definition 1.1.3]{HC} in which each of the following holds.

(1) A map is a (trivial) cofibration if and only if it is an inflation with a (trivially) cofibrant cone.

(2) A map is a (trivial) fibration if and only if it is a deflation with a (trivially) fibrant cocone.

The notion of an abelian model structure introduced in \cite{HCc} and that of an exact model structure introduced in \cite{Gillespie} are both special cases of the concept of an admissible model structure.

A triple $(\mathcal{C}, \mathcal{W}, \mathcal{F})$
of $\mathcal{B}$ is called a {\it Hovey triple} if $(\mathcal{C}\cap \mathcal{W}, \mathcal{F})$ and $(\mathcal{C}, \mathcal{W}\cap \mathcal{F})$ are complete cotorsion pairs, and $\mathcal{W}$ is {\it thick} in $\mathcal{B}$ (i.e., $\mathcal{W}$ is closed under direct summands, and if two out of three objects in a conflation are in $\mathcal{W}$, then so is the third one). A Hover triple is called {\it hereditary} if these two complete cotorsion pairs $(\mathcal{C}\cap \mathcal{W}, \mathcal{F})$ and $(\mathcal{C}, \mathcal{W}\cap \mathcal{F})$ are hereditary.

It was shown by Hovey \cite{HCc} that there exists a translation between abelian model structures and Hovey triples within an abelian category. This result was subsequently generalized by Gillespie \cite{Gillespie} to the setting of WIC exact categories, and later by Nakaoka and Palu \cite{NP} to WIC extriangulated categories. More precisely, given a WIC extriangulated category $\mathcal{B}$, one obtains a one-to-one correspondence between admissible model structures and Hovey twin cotorsion pairs. Hence, we may safely regard Hovey triples and admissible model structures as synonymous whenever no confusion is likely to arise. Further details can be found in \cite[Section 5]{NP}, \cite[Section 3.2]{Palu}, and \cite[Section 2.5]{LGZ}.

Recall from \cite{NP} that $\mathcal{B}$ is said to be \emph{Frobenius} if $\mathcal{B}$ has enough projectives, enough injectives, and the projectives coincide with the injectives.
 In this case, one has the quotient category $\underline{\mathcal{B}}$ of $\mathcal{B}$ by projectives, which is a triangulated category by \cite[Corollary 7.4 and Remark 7.5]{NP}. We refer to this category as the \emph{stable category} of $\underline{\mathcal{B}}$. Denote by  ${\rm Ho}(\mathcal{M})$ the homotopy category associated with a Hovey triple $\mathcal{M}=(\mathcal{C}, \mathcal{W}, \mathcal{F})$ in $\mathcal{B}$. Thus we have the following theorem.

\begin{thm}\label{thm0} {\rm(\cite[Sections 5-7]{NP})} Let $\mathcal{M}=(\mathcal{C}, \mathcal{W}, \mathcal{F})$ be a hereditary Hovey triple in $\mathcal{B}$. Then $\mathcal{C}\cap\mathcal{F}$ is a Frobenius extriangulated category whose class of projective-injective objects is $\mathcal{C}\cap\mathcal{W}\cap\mathcal{F}$, and there is a triangle equivalence $\underline{\mathcal{C}\cap\mathcal{F}}\cong {\rm Ho}(\mathcal{M})$.
\end{thm}
\begin{proof} Since $\mathcal{M}=(\mathcal{C}, \mathcal{W}, \mathcal{F})$ is a hereditary Hovey triple in $\mathcal{B}$, $\mathcal{C}\cap\mathcal{F}$ is a Frobenius extriangulated category whose class of projective-injective objects is $\mathcal{C}\cap\mathcal{W}\cap\mathcal{F}$, and therefore $\underline{\mathcal{C}\cap\mathcal{F}}$ is a triangulated category by \cite[Corollary 7.4]{NP}. Note that $\mathcal{M}$ is Hovey triple in $\mathcal{B}$ if and only if  $((\mathcal{C}\cap\mathcal{W},\mathcal{F}),(\mathcal{C}, \mathcal{W}\cap\mathcal{F}))$ forms a Hovey twin cotorsion pair as defined in \cite[Definition 5.1]{NP}. It follows from \cite[Theorem 6.20]{NP} that ${\rm Ho}(\mathcal{M})$ is also a triangulated category. Consequently, a direct verification shows that the equivalence $\underline{\mathcal{C}\cap\mathcal{F}}\cong {\rm Ho}(\mathcal{M})$ established in \cite[Corollary 5.25]{NP} is a triangle equivalence.
\end{proof}

\subsection{$\xi$-$\mathcal{G}$injective objects} Let us recall some definitions and basic properties of $\xi$-$\mathcal{G}$injective objects from \cite{HZZ}.

\begin{definition} \emph{(\cite[Definition 4.4]{HZZ})}
{\rm A {\it $\xi$-exact} complex $\mathbf{X}$ is an $\mathbb{E}$-triangle sequence $(\mathbf{X}, Z_\bullet(\mathbf{X}))$
 $$\mathbf{X}:=\quad\xymatrix@C=2em{\cdots\ar[r]&X_{n+1}\ar[r]^{d_{n+1}}&X_n\ar[r]^{d_n}&X_{n-1}\ar[r]&\cdots}$$
 in $\mathcal{B}$ such that for each integer $n$, the $\mathbb{E}$-triangle $\xymatrix@C=2em{Z_{n+1}(\mathbf{X})\ar[r]^{g_n}&X_n\ar[r]^{f_n}&Z_n(\mathbf{Y})\ar@{-->}[r]^{\delta_n}&}$ is in $\xi$.
}\end{definition}

\begin{definition} \emph{(\cite[Definition 4.5]{HZZ})}
{\rm Let $\mathcal{W}$ be a class of objects in $\mathcal{B}$. An $\mathbb{E}$-triangle
$$\xymatrix@C=2em{A\ar[r]& B\ar[r]& C\ar@{-->}[r]& }$$ in $\xi$ is called to be
{\it $\mathcal{B}(\mathcal{W}, -)$-exact} if for any $W\in\mathcal{W}$, the induced sequence of abelian groups $\xymatrix@C=2em{0\ar[r]&\mathcal{B}(W, A)\ar[r]&\mathcal{B}(W, B)\ar[r]&\mathcal{B}(W, C)\ar[r]& 0}$ is exact in ${\rm Ab}$}.
\end{definition}

\begin{definition} \emph{(\cite[Definition 4.6]{HZZ})}
 {\rm Let $\mathcal{W}$ be a class of objects in $\mathcal{B}$. A complex $\mathbf{X}$ is called {\it $\mathcal{B}(\mathcal{W}, -)$-exact}  if it is a $\xi$-exact complex
$$\xymatrix@C=2em{\cdots\ar[r]&X_1\ar[r]^{d_1}&X_0\ar[r]^{d_0}&X_{-1}\ar[r]&\cdots}$$ in $\mathcal{C}$ such that  there is a $\mathcal{B}(\mathcal{W},-)$-exact $\mathbb{E}$-triangle $$\xymatrix@C=2em{Z_{n+1}(\mathbf{X})\ar[r]^{g_n}&X_n\ar[r]^{f_n}&Z_n(\mathbf{X})\ar@{-->}[r]^{\delta_n}&}$$ in $\xi$ for each integer $n$.

 A $\xi$-exact complex $\mathbf{X}$ is called {\it complete $\mathcal{I}(\xi)$-exact} if it is $\mathcal{B}(\mathcal{I}(\xi), -)$-exact.}
\end{definition}

\begin{definition} \emph{(\cite[Definition 4.7]{HZZ})}
 {\rm A  {\it complete $\xi$-injective coresolution}  is a complete $\mathcal{I}(\xi)$-exact complex $$\xymatrix@C=2em{\mathbf{I}:\cdots\ar[r]&I_1\ar[r]^{d_1}&I_0\ar[r]^{d_0}&I_{-1}\ar[r]&\cdots}$$ in $\mathcal{B}$ such that $I_n$ is $\xi$-injective for each integer $n$.}
\end{definition}

\begin{definition} \emph{(\cite[Definition 4.8]{HZZ})}
{\rm  Let $\mathbf{I}$ be a complete $\xi$-injective coresolution in $\mathcal{B}$. So for each integer $n$, there exists a $\mathcal{B}(\mathcal{I}(\xi), -)$-exact $\mathbb{E}$-triangle $\xymatrix@C=2em{Z_{n+1}(\mathbf{I})\ar[r]^{g_n}&I_n\ar[r]^{f_n}&Z_n(\mathbf{I})\ar@{-->}[r]^{\delta_n}&}$ in $\xi$. The objects $Z_n(\mathbf{I})$ are called {\it $\xi$-$\mathcal{G}$injective} for each integer $n$. We denote by $\mathcal{GI}(\xi)$ the class of $\xi$-$\mathcal{G}$injective  objects.
It is obvious that $\mathcal{I(\xi)}$ $\subseteq$ $\mathcal{GI}(\xi)$.}
\end{definition}


\section{\bf Proof of Theorem \ref{thm2}}\label{section3}
This section is devoted to a proof of the statements of Theorem \ref{thm2}. We start with the following observation.

\begin{lem}\label{lemmma:3.1} Assume that $(\mathcal{X}, \mathcal{Y})$ is a complete and hereditary cotorsion pair in $(\mathcal{B}, \mathbb{E}, \mathfrak{s})$. If  $(^\perp\mathcal{Y}_n, \mathcal{Y}_n)$ is a complete cotorsion pair for some non-negative integer $n$, then $(^\perp\mathcal{Y}_n, \mathcal{Y}_n)$ is hereditary.
\end{lem}
\begin{proof}
  The proof is straightforward by Lemma \ref{lem2.18} and Proposition \ref{cor2.20}.
\end{proof}


%
Now we give a description of $\mathcal{F}_n$ for a hereditary Hovey triple $(\mathcal{C}, \mathcal{W}, \mathcal{F})$.

\begin{thm}\label{thm1} Let $(\mathcal{C}, \mathcal{W}, \mathcal{F})$ be a hereditary Hovey triple in $(\mathcal{B}, \mathbb{E}, \mathfrak{s})$. Consider the following conditions for  any $M\in\mathcal{B}$ and  any  non-negative integer $n$:

$(1)$ $M\in\mathcal{F}_n$;

$(2)$ There is an $\mathbb{E}$-triangle $\xymatrix@=2em{N\ar[r]&L\ar[r]&M\ar@{-->}[r]&}$ with $N\in\mathcal{F}$ and $L\in(\mathcal{W}\cap \mathcal{F})_{n}$;

$(3)$ There is an $\mathbb{E}$-triangle $\xymatrix@=2em{M\ar[r]&F\ar[r]&K\ar@{-->}[r]&}$ with $F\in\mathcal{F}$ and $K\in(\mathcal{W}\cap \mathcal{F})_{n-1}$;

$(4)$ $\mathbb{E}(H, M)=0$ for every $H\in{^\perp(\mathcal{W}\cap \mathcal{F})}_n\cap \mathcal{W}$.

Then $(1)\Leftrightarrow (2)\Leftrightarrow (3)\Rightarrow (4)$. The converse holds if $(^\perp(\mathcal{W}\cap\mathcal{F})_n, (\mathcal{W}\cap\mathcal{F})_n)$ is a complete cotorsion pair.
\end{thm}

\begin{proof} $(1)\Rightarrow (2).$ Assume that $M\in\mathcal{F}_n$. We will proceed by induction on $n$. For $n=0$, there is an $\mathbb{E}$-triangle $\xymatrix{N\ar[r]&L\ar[r]&M\ar@{-->}[r]&}$ with $L\in\mathcal{C}\cap \mathcal{W}$ and $N\in\mathcal{F}$ as $(\mathcal{C}\cap\mathcal{W}, \mathcal{F})$ is a complete cotorsion pair. Since $\mathcal{F}$ is closed under extensions and $M, N\in\mathcal{F}$, we have $L\in\mathcal{C}\cap\mathcal{W}\cap\mathcal{F}\subseteq \mathcal{W}\cap\mathcal{F}$.

Now assume $n\geqslant 1$, there is an $\mathbb{E}$-triangle $\xymatrix{M\ar[r]&F\ar[r]&K\ar@{-->}[r]&}$ with $K\in\mathcal{C}\cap \mathcal{W}$ and $F\in\mathcal{F}$ since $(\mathcal{C}\cap\mathcal{W}, \mathcal{F})$ is a complete cotorsion pair. Hence $K\in\mathcal{F}_{n-1}$ by Lemma \ref{lem2.19}. So the induction hypothesis yields an $\mathbb{E}$-triangle $\xymatrix{N_1\ar[r]&L_1\ar[r]&K\ar@{-->}[r]&}$ with $L_1\in(\mathcal{W}\cap \mathcal{F})_{n-1}$ and $N_1\in\mathcal{F}$. It follows from Lemma \ref{lem1} that there is a commutative diagram with $\mathbb{E}$-triangles
$$\xymatrix{&M\ar[d]\ar@{=}[r]&M\ar[d]&\\
N_1\ar@{=}[d]\ar[r]&D\ar[r]\ar[d]&F\ar@{-->}[r]\ar[d]&\\
N_1\ar[r]&L_1\ar@{-->}[d]\ar[r]&K\ar@{-->}[d]\ar@{-->}[r]&\\
&&&}$$
where $D\in\mathcal{F}$ since $F, N_1\in\mathcal{F}$ and $\mathcal{F}$ is closed under extensions. By the case where $n=0$, there is an $\mathbb{E}$-triangle
 $\xymatrix{N\ar[r]&L_2\ar[r]&D\ar@{-->}[r]&}$ with $L_2\in\mathcal{W}\cap \mathcal{F}$ and $N\in\mathcal{F}$. Then there is a commutative diagram with $\mathbb{E}$-triangles
  $$\xymatrix{N\ar[r]\ar@{=}[d]&L\ar[r]\ar[d]&M\ar@{-->}[r]\ar[d]&\\
N\ar[r]&L_2\ar[r]\ar[d]&D\ar@{-->}[r]\ar[d]&\\
&L_1\ar@{=}[r]\ar@{-->}[d]&L_1\ar@{-->}[d]&\\
&&&}$$
by $({\rm ET}4)^{\rm op}$. Hence $L\in(\mathcal{W}\cap\mathcal{F})_n$ as $L_2\in\mathcal{W}\cap\mathcal{F}$ and $L_1\in(\mathcal{W}\cap\mathcal{F})_{n-1}$. Then the top row in the above diagram is the required $\mathbb{E}$-triangle.

$(2)\Rightarrow (3).$ Assume that $\xymatrix@=2em{N\ar[r]&L\ar[r]&M\ar@{-->}[r]&}$ is an $\mathbb{E}$-triangle with $L\in(\mathcal{W}\cap \mathcal{F})_{n}$ and $N\in\mathcal{F}$. So there is an $\mathbb{E}$-triangle $\xymatrix@=2em{L\ar[r]&H\ar[r]&K\ar@{-->}[r]&}$  with $H\in\mathcal{W}\cap\mathcal{F}$ and $K\in(\mathcal{W}\cap \mathcal{F})_{n-1}$. It follows from {\rm (ET4)} that there is a commutative diagram with $\mathbb{E}$-triangles
 $$\xymatrix{N\ar[r]\ar@{=}[d]&L\ar[r]\ar[d]&M\ar@{-->}[r]\ar[d]&\\
N\ar[r]&H\ar[r]\ar[d]&F\ar@{-->}[r]\ar[d]&\\
&K\ar@{=}[r]\ar@{-->}[d]&K\ar@{-->}[d]&\\
&&&}$$
Since the cotorsion pair $(\mathcal{C}\cap\mathcal{W}, \mathcal{F})$ is hereditary, $\mathcal{F}$ is closed under cones of inflations by Lemma \ref{lem2.18}, one can obtain $F\in\mathcal{F}$ since $N, H\in\mathcal{F}$. Hence the right  column in the above diagram is the required $\mathbb{E}$-triangle.

$(3)\Rightarrow (1)$ is obvious.

$(2)\Rightarrow (4).$ Assume that $\xymatrix@=2em{N\ar[r]&L\ar[r]&M\ar@{-->}[r]&}$ is an $\mathbb{E}$-triangle with $L\in(\mathcal{W}\cap \mathcal{F})_{n}$ and $N\in\mathcal{F}$. Hence there is an $\mathbb{E}$-triangle $\xymatrix@=2em{L\ar[r]&D\ar[r]&K\ar@{-->}[r]&}$  with $K\in(\mathcal{W}\cap \mathcal{F})_{n-1}$ and $D\in\mathcal{W}\cap\mathcal{F}$. It follows from (ET4) that there is a commutative diagram with $\mathbb{E}$-triangles
$$\xymatrix{N\ar[r]\ar@{=}[d]&L\ar[r]\ar[d]&M\ar@{-->}[r]\ar[d]&\\
N\ar[r]&D\ar[r]\ar[d]&F\ar@{-->}[r]\ar[d]&\\
&K\ar@{=}[r]\ar@{-->}[d]&K\ar@{-->}[d]&\\
&&&}$$
Hence $F\in\mathcal{F}$ as $N, D\in\mathcal{F}$ and $\mathcal{F}$ is closed under cones of inflations. Since ${^\perp(\mathcal{W}\cap \mathcal{F})}_n\cap\mathcal{W}\subseteq {^\perp(\mathcal{W}\cap \mathcal{F})}\cap\mathcal{W}=\mathcal{C}\cap\mathcal{W}={^\perp\mathcal{F}}$, we have $\mathbb{E}(H, F)=0$ for any $H\in{^\perp(\mathcal{W}\cap \mathcal{F})}_n\cap\mathcal{W}$. Moreover, we have the following commutative diagram with exact rows
$$\xymatrix{\mathcal{B}(H, L)\ar[d]\ar[r]&\mathcal{B}(H, D)\ar[d]^{f}\ar[r]^{h}&\mathcal{B}(H, K)\ar@{=}[d]\ar[r]&\mathbb{E}(H, L)=0\ar[d]&\\
\mathcal{B}(H, M)\ar[r]&\mathcal{B}(H, F)\ar[r]^g&\mathcal{B}(H, K)\ar[r]&\mathbb{E}(H, M)\ar[r]&\mathbb{E}(H, F)=0}$$
The surjectivity of $h$, together with the equality $h = g f$, implies that $g$ is surjective. Moreover, we have $\mathbb{E}(H, F) = 0$. From the exactness of the second row in the preceding commutative diagram, it follows that $\mathbb{E}(H, M) = 0$, as required.

$(4)\Rightarrow (2).$ Assume that $\mathbb{E}(H, M)=0$ for any $H\in{^\perp(\mathcal{W}\cap \mathcal{F})}_n\cap\mathcal{W}$. Since the cotorsion pair $(\mathcal{C}\cap \mathcal{W}, \mathcal{F})$ is complete, there is an $\mathbb{E}$-triangle $\xymatrix@=2em{N\ar[r]&L\ar[r]&M\ar@{-->}[r]&}$ with $L\in\mathcal{C}\cap\mathcal{W}$ and $N\in\mathcal{F}$. It suffices to show that $L\in(\mathcal{W}\cap\mathcal{F})_n$. For any $H\in{^\perp(\mathcal{W}\cap \mathcal{F})}_n\cap\mathcal{W}$, we have $\mathbb{E}(H, N)=0$ since ${^\perp(\mathcal{W}\cap \mathcal{F})}_n\cap\mathcal{W}\subseteq {^\perp(\mathcal{W}\cap \mathcal{F})}\cap\mathcal{W}=\mathcal{C}\cap\mathcal{W}={^\perp\mathcal{F}}$. Therefore, the exact sequence of abelian groups $\xymatrix{0=\mathbb{E}(H, N)\ar[r]& \mathbb{E}(H, L)\ar[r]& \mathbb{E}(H, M)=0}$ implies that $\mathbb{E}(H, L)=0$. This shows that $L\in({^\perp(\mathcal{W}\cap \mathcal{F})}_n\cap\mathcal{W})^\perp$. Since the cotorsion pair $(^\perp(\mathcal{W}\cap\mathcal{F})_n, (\mathcal{W}\cap\mathcal{F})_n)$ is complete by hypothesis, there is an $\mathbb{E}$-triangle $\xymatrix{L\ar[r]& L_1\ar[r]& K\ar@{-->}[r]&}$ with $L_1\in(\mathcal{W}\cap\mathcal{F})_n$ and $K\in{^\perp(\mathcal{W}\cap\mathcal{F})}_n$. It is easy to check that $L_1\in(\mathcal{W}\cap\mathcal{F})_n\subseteq \mathcal{W}$ since $\mathcal{W}$ is thick. Note that $L\in\mathcal{C}\cap\mathcal{W}\subseteq\mathcal{W}$. One can obtain $K\in \mathcal{W}$ by the thickness of $\mathcal{W}$, and $K\in{^\perp(\mathcal{W}\cap \mathcal{F})}_n\cap\mathcal{W}$. Hence the $\mathbb{E}$-triangle $\xymatrix{L\ar[r]& L_1\ar[r]& K\ar@{-->}[r]&}$ is split. Therefore $L\in(\mathcal{W}\cap\mathcal{F})_n$, as desired.
\end{proof}
\begin{cor}\label{cor1}Let $(\mathcal{C}, \mathcal{W}, \mathcal{F})$ be a hereditary Hovey triple in $(\mathcal{B}, \mathbb{E}, \mathfrak{s})$. Then $\mathcal{W}\cap\mathcal{F}_n=(\mathcal{W}\cap\mathcal{F})_n$ for any non-negative integer $n$.
\end{cor}
\begin{proof} It is obvious that $(\mathcal{W}\cap \mathcal{F})_n\subseteq \mathcal{F}_n$. Since $\mathcal{W}$ is thick, it is easy to prove that $(\mathcal{W}\cap \mathcal{F})_n\subseteq \mathcal{W}$. Hence $(\mathcal{W}\cap \mathcal{F})_n\subseteq \mathcal{W}\cap\mathcal{F}_n$. Conversely, if $M\in\mathcal{W}\cap\mathcal{F}_n$, then $M\in\mathcal{F}_n$, and there is an $\mathbb{E}$-triangle $\xymatrix{M\ar[r]&F\ar[r]&K\ar@{-->}[r]&}$ with $F\in\mathcal{F}$ and $K\in(\mathcal{W}\cap\mathcal{F})_{n-1}$ by Theorem \ref{thm1}. Since $\mathcal{W}$ is thick, it is easy to see that $F\in\mathcal{W}$ as $M\in\mathcal{W}$ and $K\in(\mathcal{W}\cap\mathcal{F})_{n-1}\subseteq \mathcal{W}$. Hence $M\in(\mathcal{W}\cap\mathcal{F})_{n}$ as $F\in\mathcal{W}\cap\mathcal{F}$ and $K\in(\mathcal{W}\cap\mathcal{F})_{n-1}$, so we have $\mathcal{W}\cap\mathcal{F}_n\subseteq (\mathcal{W}\cap\mathcal{F})_{n}$. Therefore, $\mathcal{W}\cap\mathcal{F}_n=(\mathcal{W}\cap\mathcal{F})_n$.
\end{proof}

We are now in a position to prove Theorem \ref{thm2} in the introduction.

{\bf Proof of Theorem \ref{thm2}.} The case $n=0$ is obvious. Now assume $n>0$.

It follows from Theorem \ref{thm1} that $\mathcal{F}_n=(^\perp(\mathcal{W}\cap \mathcal{F})_n\cap\mathcal{W})^\perp$ and ${^\perp(\mathcal{W}\cap \mathcal{F})}_n\cap\mathcal{W}\subseteq {^\perp\mathcal{F}}_n$. In addition,
$$^\perp\mathcal{F}_n\subseteq ^\perp((\mathcal{W}\cap \mathcal{F})_n\cup\mathcal{F})={^\perp(\mathcal{W}\cap \mathcal{F})}_n\cap{^\perp\mathcal{F}}={^\perp(\mathcal{W}\cap \mathcal{F})}_n\cap\mathcal{C}\cap\mathcal{W}\subseteq{^\perp(\mathcal{W}\cap \mathcal{F})}_n\cap\mathcal{W}.$$
Thus ${^\perp(\mathcal{W}\cap \mathcal{F})}_n\cap\mathcal{W}={^\perp\mathcal{F}}_n$, and $({^\perp(\mathcal{W}\cap \mathcal{F})}_n\cap\mathcal{W}, \mathcal{F}_n)$ is a cotorsion pair. By Corollary \ref{cor1}, $\mathcal{W}\cap\mathcal{F}_n=(\mathcal{W}\cap\mathcal{F})_n$, then the core of $({^\perp(\mathcal{W}\cap \mathcal{F})}_n\cap\mathcal{W}, \mathcal{F}_n)$ is ${^\perp(\mathcal{W}\cap \mathcal{F})}_n\cap(\mathcal{W}\cap\mathcal{F})_n$.

Next we show that the cotorsion pair $({^\perp(\mathcal{W}\cap \mathcal{F})}_n\cap\mathcal{W}, \mathcal{F}_n)$ is complete. For any $M\in\mathcal{B}$, there is an $\mathbb{E}$-triangle $\xymatrix{M\ar[r]&H\ar[r]&N\ar@{-->}[r]&}$ with $N\in\mathcal{C}\cap\mathcal{W}$ and $H\in\mathcal{F}$ by the completeness of cotorsion pair $(\mathcal{C}\cap \mathcal{W}, \mathcal{F})$. Since the cotorsion pair $(^\perp(\mathcal{W}\cap\mathcal{F})_n, (\mathcal{W}\cap\mathcal{F})_n)$ is complete, there is an $\mathbb{E}$-triangle $\xymatrix{B\ar[r]&L_1\ar[r]&N\ar@{-->}[r]&}$ with $L_1\in{^\perp(\mathcal{W}\cap\mathcal{F})}_n$ and $B\in (\mathcal{W}\cap\mathcal{F})_n$. It follows from Lemma \ref{lem1} that there is a commutative with $\mathbb{E}$-triangles
$$\xymatrix{&B\ar[d]\ar@{=}[r]&B\ar[d]&\\
M\ar@{=}[d]\ar[r]&D\ar[r]\ar[d]&L_1\ar@{-->}[r]\ar[d]&\\
M\ar[r]&H\ar@{-->}[d]\ar[r]&N\ar@{-->}[d]\ar@{-->}[r]&\\
&&&}$$
Since $H\in\mathcal{F}\subseteq \mathcal{F}_n$ and $B\in(\mathcal{W}\cap\mathcal{F})_n\subseteq \mathcal{F}_n$, we have $D\in\mathcal{F}_n$ by Proposition \ref{cor2.20}(1). Note that $N\in\mathcal{C}\cap\mathcal{W}\subseteq\mathcal{W}$ and $B\in(\mathcal{W}\cap\mathcal{F})_n=\mathcal{W}\cap\mathcal{F}_n\subseteq\mathcal{W}$. One can obtain $L_1\in\mathcal{W}$ as $\mathcal{W}$ is thick. Therefore, $\xymatrix{M\ar[r]&D\ar[r]&L_1\ar@{-->}[r]&}$ is an $\mathbb{E}$-triangle with $L_1\in{^\perp(\mathcal{W}\cap\mathcal{F})}_n\cap\mathcal{W}$ and $D\in\mathcal{F}_n$.

On the other hand, we have an $\mathbb{E}$-triangle $\xymatrix{L\ar[r]&K\ar[r]&M\ar@{-->}[r]&}$ with $L\in\mathcal{F}$ and $K\in\mathcal{C}\cap\mathcal{W}$ by the completeness of cotorsion pair $(\mathcal{C}\cap \mathcal{W}, \mathcal{F})$. Since the cotorsion pair $(^\perp(\mathcal{W}\cap\mathcal{F})_n, (\mathcal{W}\cap\mathcal{F})_n)$ is complete, there is an $\mathbb{E}$-triangle $\xymatrix{F\ar[r]&K_1\ar[r]&K\ar@{-->}[r]&}$ with $K_1\in{^\perp(\mathcal{W}\cap\mathcal{F})}_n$ and $F\in (\mathcal{W}\cap\mathcal{F})_n$. It follows from (ET4)$^{\rm op}$ that there is a commutative diagram with $\mathbb{E}$-triangles
 $$\xymatrix{F\ar[r]\ar@{=}[d]&Y\ar[r]\ar[d]&L\ar@{-->}[r]\ar[d]&\\
F\ar[r]&K_1\ar[r]\ar[d]&K\ar@{-->}[r]\ar[d]&\\
&M\ar@{=}[r]\ar@{-->}[d]&M\ar@{-->}[d]&\\
&&&}$$
Since $L\in\mathcal{F}\subseteq \mathcal{F}_n$ and $F\in(\mathcal{W}\cap\mathcal{F})_n\subseteq \mathcal{F}_n$, we have $Y\in\mathcal{F}_n$ by Proposition \ref{cor2.20}(1). Note that $K\in\mathcal{C}\cap\mathcal{W}\subseteq\mathcal{W}$ and $F\in(\mathcal{W}\cap\mathcal{F})_n=\mathcal{W}\cap\mathcal{F}_n\subseteq\mathcal{W}$. One can obtain $K_1\in\mathcal{W}$ as $\mathcal{W}$ is thick. Then $\xymatrix{Y\ar[r]&K_1\ar[r]&M\ar@{-->}[r]&}$ is an $\mathbb{E}$-triangle with $K_1\in{^\perp(\mathcal{W}\cap\mathcal{F})}_n\cap\mathcal{W}$ and $Y\in\mathcal{F}_n$. Therefore, the cotorsion pair $({^\perp(\mathcal{W}\cap \mathcal{F})}_n\cap\mathcal{W}, \mathcal{F}_n)$ is complete.

Finally, $({^\perp(\mathcal{W}\cap \mathcal{F})}_n\cap\mathcal{W}, \mathcal{F}_n)$ is also hereditary by Lemma \ref{lemmma:3.1}.

(2) It follows from Lemma \ref{lemmma:3.1} that the complete cotorsion pair $({^\perp(\mathcal{W}\cap \mathcal{F})}_n, (\mathcal{W}\cap\mathcal{F})_n)$ is
hereditary. Since $\mathcal{W}$ is thick, we obtain a hereditary Hovey triple $\mathcal{M}_n = ({^\perp(\mathcal{W}\cap \mathcal{F})}_n, \mathcal{W}, \mathcal{F}_n)$.
Moreover, ${^\perp(\mathcal{W}\cap \mathcal{F})}_n\cap\mathcal{F}_n$ is a Frobenius category category with the class of projective-injective objects ${^\perp(\mathcal{W}\cap \mathcal{F})}_n\cap(\mathcal{W}\cap \mathcal{F})_n$.

 The remaining assertions hold by Theorem \ref{thm0}.\hfill$\Box$

A direct consequence of Theorem \ref{thm2} yields the following result, which improves \cite[Theorem 3.6]{SWZ} and \cite[Theorem A]{El Maaouy2}.

\begin{cor}\label{cor3.4} Let $\mathcal{B}$ be a WIC exact category, and $\mathcal{M}=(\mathcal{C}, \mathcal{W}, \mathcal{F})$  a hereditary Hovey triple in $\mathcal{B}$. Assume that $(^\perp(\mathcal{W}\cap\mathcal{F})_n, (\mathcal{W}\cap\mathcal{F})_n)$ is a complete cotorsion pair for some non-negative integer $n$. Then

$(1)$ The pair $({^\perp(\mathcal{W}\cap \mathcal{F})}_n\cap\mathcal{W}, \mathcal{F}_n)$
is a complete and hereditary cotorsion pair with core $^\perp(\mathcal{W}\cap\mathcal{F})_n\cap (\mathcal{W}\cap\mathcal{F})_n$.

$(2)$ The triple $\mathcal{M}_n=({^\perp(\mathcal{W}\cap \mathcal{F})}_n, \mathcal{W}, \mathcal{F}_n)$ is a hereditary Hovey triple in $\mathcal{A}$, and ${^\perp(\mathcal{W}\cap \mathcal{F})}_n\cap\mathcal{F}_n$ is a Frobenius category with the class of projective-injective objects ${^\perp(\mathcal{W}\cap \mathcal{F})}_n\cap(\mathcal{W}\cap \mathcal{F})_n$. Furthermore, we obtain triangle equivalences
$$\underline{{^\perp(\mathcal{W}\cap \mathcal{F})}_n\cap\mathcal{F}_n}\simeq {\rm Ho}(\mathcal{M}_n)= {\rm Ho}(\mathcal{M})\simeq \underline{\mathcal{C}\cap\mathcal{F}}.$$
\end{cor}

\begin{rem}\label{rem:3.5} (1) Note that Corollary \ref{cor3.4} has been proved by Shao-Wang-Zhang in \cite[Theorem 3.6]{SWZ} under the additional condition that the WIC exact category has enough projectives and enough injectives.

(2)  Note that every abelian category is a WIC exact category. Therefore, Corollary \ref{cor3.4} refines \cite[Theorem A]{El Maaouy2} by eliminating the superfluous hypotheses. More precisely, the author constructed the hereditary Hovey triple $\mathcal{M}_n=({^\perp(\mathcal{W}\cap \mathcal{F})}_n, \mathcal{W}, \mathcal{F}_n)$ under the assumptions that $\mathcal{B}$ is an abelian category and that $(\mathcal{C}, \mathcal{W}, \mathcal{F})$ is a hereditary Hovey triple on $\mathcal{B}$ satisfying the following conditions:

\begin{enumerate}
    \item[(i)] $(\mathcal{W},\mathcal{W}^\perp)$ is a complete cotorsion pair.
    \item[(ii)] $(^\perp(\mathcal{W}\cap\mathcal{F})_k, (\mathcal{W}\cap\mathcal{F})_k)$ is a complete cotorsion pair for every integer $k\in\{0,1,\dots,n\}$.
\end{enumerate}

\end{rem}

%
%
%
%
%
%
\section{\bf Model structures arising from $\mathcal{GI}(\xi)_n$}\label{model-sts-arsing-from-GI}\label{section4}

This section is devoted to studying model structures induced by objects of finite $\xi$-$\mathcal{G}$injective dimension, and to proving Theorem \ref{thm4}. In what follows, we always assume that $\xi$ is a proper class in $(\mathcal{B}, \mathbb{E}, \mathfrak{s})$.  It follows from \cite[Theorem 3.2]{HZZ} that $(\mathcal{B}, \mathbb{E}_\xi, \mathfrak{s}_\xi)$ is an extriangulated category with $\mathbb{E}_\xi:=\mathbb{E}|_\xi$ and $\mathfrak{s}_\xi:=\mathfrak{s}|_{\mathbb{E}_\xi}$, where  $$\mathbb{E}_\xi(C, A)=\{\delta\in\mathbb{E}(C, A)~|~\delta~ \textrm{is realized as an $\mathbb{E}$-triangle}\xymatrix{A\ar[r]^x&B\ar[r]^y&C\ar@{-->}[r]^{\delta}&}~\textrm{in}~\xi\}$$ for any $A, C\in\mathcal{C}$.

\begin{lem}\label{lemma4.1} Let $\xi$ be a proper class in an extriangulated category $(\mathcal{B}, \mathbb{E}, \mathfrak{s})$.

$(1)$  $\xymatrix{A\ar[r]^x&B\ar[r]^y&C\ar@{-->}[r]^{\delta}&}$ is an $\mathbb{E}$-triangle in $\xi$ if and only if it is an $\mathbb{E}_\xi$-triangle in the extriangulated category $(\mathcal{B}, \mathbb{E}_\xi, \mathfrak{s}_\xi)$.

$(2)$ If $(\mathcal{B}, \mathbb{E}, \mathfrak{s})$ is a weakly idempotent complete extriangulated category with enough $\xi$-injectives and $I$ is an object in $\mathcal{B}$, then  $I$ is $\xi$-injective if and only if $\mathbb{E}_\xi(B, I)=0$ for any $B\in\mathcal{B}$.

$(3)$ If $(\mathcal{B}, \mathbb{E}, \mathfrak{s})$ has enough $\xi$-projectives and enough $\xi$-injectives, then $(\mathcal{B}, \mathbb{E}_\xi, \mathfrak{s}_\xi)$ has enough projectives and enough injectives.
\end{lem}
\begin{proof} (1) is obvious.

(2) Assume that $I$ is $\xi$-injective in $(\mathcal{B}, \mathbb{E}, \mathfrak{s})$ . It is easy to check that $\mathbb{E}_\xi(B, I)=0$ for any object $B$ in $\mathcal{B}$. Conversely, if $\mathbb{E}_\xi(B, I)=0$ for any $B\in\mathcal{B}$. Since $(\mathcal{B}, \mathbb{E}, \mathfrak{s})$ has enough $\xi$-injectives, there is an $\mathbb{E}_\xi$-triangle $\xymatrix{I\ar[r]&E\ar[r]&K\ar@{-->}[r]&}$  with $E\in\mathcal{I}(\xi)$ and $K\in\mathcal{B}$. It is easy to see that this $\mathbb{E}_\xi$-triangle splits as $\mathbb{E}_\xi(K, I)=0$ by hypothesis, which implies that $I\in\mathcal{I}(\xi)$.

(3) It is easy to obtain by (2) and its dual version.
\end{proof}

In the following, we always assume that $(\mathcal{B}, \mathbb{E}, \mathfrak{s})$ is a weakly idempotent complete extriangulated category with enough $\xi$-projectives and enough $\xi$-injectives, that is, $(\mathcal{B}, \mathbb{E}_\xi, \mathfrak{s}_\xi)$ is a weakly idempotent complete extriangulated category with enough projectives and enough injectives. Let $\mathcal{X}$ be a class of objects of $\mathcal{B}$. Put
$$\mathcal{X}^\perp=\{Z\in\mathcal{B}\mid \mathbb{E}_\xi(X, Z)=0, \ \forall \ X\in\mathcal{X}\} \ \mbox{and} \ ^\perp\mathcal{X}=\{Z\in\mathcal{B}\mid \mathbb{E}_\xi(Z, X)=0, \ \forall \ X\in\mathcal{X}\}.$$

The following result is essentially taken from \cite[Lemma 3.1]{GLZ}, we obtain it from the dual proof given there.

\begin{lem}\label{lem4.2} 
$^\perp\mathcal{GI}(\xi)$ is thick in $(\mathcal{B}, \mathbb{E}_\xi, \mathfrak{s}_\xi)$ and
$\mathcal{GI}(\xi)\cap{^\perp\mathcal{GI}}(\xi)=\mathcal{I}(\xi)$.
\end{lem}

\begin{lem}\label{lem4.3} 
$\mathcal{I}(\xi)_n\subseteq {^\perp\mathcal{GI}}(\xi)$.
\end{lem}
\begin{proof} Assume $G\in\mathcal{GI}(\xi)$ and $C\in\mathcal{I}(\xi)_n$. Then there is an $\mathbb{E}_\xi$-triangle $\xymatrix@=1.8em{G\ar[r]&I\ar[r]&K\ar@{-->}[r]&}$ with $I\in\mathcal{I}(\xi)$ and $K\in\mathcal{GI}(\xi)$. Applying the functor $\mathcal{B}(C, -)$ to above $\mathbb{E}_\xi$-triangle, one has an exact sequence of abelian groups $\mathcal{B}(C, I)\rightarrow\mathcal{B}(C, K)\rightarrow \mathbb{E}_\xi(C, G)\rightarrow \mathbb{E}_\xi(C, I)=0$. Meanwhile,  there is an exact sequence of abelian groups $\mathcal{B}(C, I)\rightarrow\mathcal{B}(C, K)\rightarrow 0$ by the dual of \cite[Lemma 5.3]{HZZ}, which implies that $\mathbb{E}_\xi(C, G)=0$. Hence $\mathcal{I}(\xi)_n\subseteq {^\perp\mathcal{GI}}(\xi)$.
\end{proof}

The following result is easily obtained from the dual version of \cite[Theorem 4.16 and Proposition 5.2]{HZZ} and Horseshoe Lemma (cf. dual version of \cite[Lemma 3.3]{HZZ1}).
\begin{lem}\label{lem4.4}Let 
$n$ be a non-negative integer. Then $\mathcal{GI}(\xi)_n$ is closed under extensions and closed under cones of inflations in $(\mathcal{B}, \mathbb{E}_\xi, \mathfrak{s}_\xi)$.
\end{lem}

Now we give a description of objects in $\mathcal{GI}(\xi)_n$.

\begin{thm}\label{thm3} 
For  any $M\in\mathcal{B}$ and non-negative integer $n$, consider the following conditions:

$(1)$ $M\in\mathcal{GI}(\xi)_n$;

$(2)$ There is an $\mathbb{E}_\xi$-triangle $\xymatrix@=2em{N\ar[r]&L\ar[r]&M\ar@{-->}[r]&}$ with $N\in\mathcal{GI}(\xi)$ and $L\in\mathcal{I}(\xi)_{n}$;

$(3)$ There is an $\mathbb{E}_\xi$-triangle $\xymatrix@=2em{M\ar[r]&G\ar[r]&K\ar@{-->}[r]&}$ with $G\in\mathcal{GI}(\xi)$ and $K\in\mathcal{I}(\xi)_{n-1}$;

$(4)$ $\mathbb{E}_\xi(H, M)=0$ for every $H\in{^\perp\mathcal{I}}(\xi)_n\cap{^\perp\mathcal{GI}}(\xi)$.

Then $(1)\Leftrightarrow (2)\Leftrightarrow (3)\Rightarrow (4)$. Moreover,  if both $(^\perp\mathcal{I}(\xi)_n, \mathcal{I}(\xi)_n)$ and $(^\perp\mathcal{GI}(\xi), \mathcal{GI}(\xi))$ are complete cotorsion pairs in $(\mathcal{B}, \mathbb{E}_\xi, \mathfrak{s}_\xi)$, then $(4)\Rightarrow (2)$ holds.
\end{thm}
\begin{proof} $(1)\Rightarrow (2)$. It follows from the dual version of \cite[Proposition 5.5]{HZZ}.

$(2)\Rightarrow (3)$. Assume that $\xymatrix@=2em{N\ar[r]&L\ar[r]&M\ar@{-->}[r]&}$ is an $\mathbb{E}_\xi$-triangle with $L\in\mathcal{I}(\xi)_{n}$ and $N\in\mathcal{GI}(\xi)$. So there is an $\mathbb{E}_\xi$-triangle $\xymatrix@=2em{L\ar[r]&I\ar[r]&K\ar@{-->}[r]&}$  with $I\in\mathcal{I}(\xi)$ and $K\in\mathcal{I}(\xi)_{n-1}$. It follows from {\rm (ET4)} that there is a commutative diagram with $\mathbb{E}_\xi$-triangles
 $$\xymatrix{N\ar[r]\ar@{=}[d]&L\ar[r]\ar[d]&M\ar@{-->}[r]\ar[d]&\\
N\ar[r]&I\ar[r]\ar[d]&G\ar@{-->}[r]\ar[d]&\\
&K\ar@{=}[r]\ar@{-->}[d]&K\ar@{-->}[d]&\\
&&&}$$
It follows from the dual version of \cite[Theorem 4.16]{HZZ} that $G\in\mathcal{GI}(\xi)$ since $N, I\in\mathcal{GI}(\xi)$. Hence the right column in the above diagram is the required $\mathbb{E}_\xi$-triangle.

$(3)\Rightarrow (1)$. It is obvious.

$(2)\Rightarrow (4)$.  Assume that $\xymatrix@=2em{N\ar[r]&L\ar[r]&M\ar@{-->}[r]&}$ is an $\mathbb{E}_\xi$-triangle with $N\in\mathcal{GI}(\xi)$ and $L\in\mathcal{I}(\xi)_{n}$. Next we claim that $\mathbb{E}_\xi^2(C, N)=0$ for any $C\in{^\perp\mathcal{GI}}(\xi)$. Since $N\in\mathcal{GI}(\xi)$, there is an
$\mathbb{E}_\xi$-triangle $\xymatrix@=2em{N\ar[r]&I\ar[r]&G\ar@{-->}[r]&}$ with $G\in\mathcal{GI}(\xi)$ and $I\in\mathcal{I}(\xi)$. Hence there is an exact sequence of abelian groups $\xymatrix@=1.8em{0=\mathbb{E}_\xi(C, G)\ar[r]&\mathbb{E}_\xi^2(C, N)\ar[r]&\mathbb{E}_\xi^2(C, I)=0}$ by Lemma \ref{lem2.11}, which implies that $\mathbb{E}_\xi^2(C, N)=0$. For any $H\in{^\perp\mathcal{I}(\xi)}_n\cap{^\perp\mathcal{GI}}(\xi)$, there is an exact sequence of abelian groups $\xymatrix@=1.8em{0=\mathbb{E}_\xi(H, L)\ar[r]&\mathbb{E}_\xi(H, M)\ar[r]&\mathbb{E}_\xi^2(H, N)=0}$ by Lemma \ref{lem2.11} again. So $\mathbb{E}_\xi(M, H)=0$, as desired.

Now assume that $(^\perp\mathcal{I}(\xi)_n, \mathcal{I}(\xi)_n)$ and $(^\perp\mathcal{GI}(\xi), \mathcal{GI}(\xi))$ are complete cotorsion pairs in $(\mathcal{B}, \mathbb{E}_\xi, \mathfrak{s}_\xi)$.

$(4)\Rightarrow (2)$. Assume that $\mathbb{E}_\xi(H, M)=0$ for every $H\in{^\perp\mathcal{I}(\xi)}_n\cap{^\perp\mathcal{GI}}(\xi)$. Then there is an $\mathbb{E}_\xi$-triangle $\xymatrix@=2em{N\ar[r]&L\ar[r]&M\ar@{-->}[r]&}$ with $L\in{^\perp\mathcal{GI}}(\xi)$ and $N\in\mathcal{GI}(\xi)$ since the cotorsion pair $(^\perp\mathcal{GI}(\xi), \mathcal{GI}(\xi))$ is complete. Next we claim that $L\in\mathcal{I}(\xi)_n$. Since $\mathbb{E}_\xi(C, N)=0$ for any $C\in{^\perp\mathcal{I}(\xi)}_n\cap{^\perp\mathcal{GI}}(\xi)$, it follows that $\mathbb{E}_\xi(C, L)=0$ for any $C\in{^\perp\mathcal{I}(\xi)}_n\cap{^\perp\mathcal{GI}}(\xi)$. There is an $\mathbb{E}_\xi$-triangle $\xymatrix@=2em{L\ar[r]&K\ar[r]&H\ar@{-->}[r]&}$ with $H\in{^\perp\mathcal{I}}(\xi)_n$ and $K\in\mathcal{I}(\xi)_n$ because the cotorsion pair $(^\perp\mathcal{I}(\xi)_n, \mathcal{I}(\xi)_n)$ is complete. Since $K\in\mathcal{I}(\xi)_n\subseteq {^\perp\mathcal{GI}}(\xi)$ by Lemma \ref{lem4.3} and $L\in{^\perp\mathcal{GI}}(\xi)$, it follows that $H\in{^\perp\mathcal{GI}}(\xi)$ because ${^\perp\mathcal{GI}}(\xi)$ is thick by Lemma \ref{lem4.2}. Hence $H\in{^\perp\mathcal{I}(\xi)}_n\cap{^\perp\mathcal{GI}}(\xi)$ and $\mathbb{E}_\xi(H, L)=0$.
Therefore, the $\mathbb{E}_\xi$-triangle $\xymatrix@=2em{L\ar[r]&K\ar[r]&H\ar@{-->}[r]&}$ splits, and $L\in\mathcal{I}(\xi)_n$.
\end{proof}

\begin{cor}\label{cor4.7}Let 
$n$ be a non-negative integer. Then $^\perp\mathcal{GI}(\xi)\cap\mathcal{GI}(\xi)_n=\mathcal{I}(\xi)_n$.
\end{cor}
\begin{proof} It is easy to see that $\mathcal{I}(\xi)_n\subseteq {^\perp\mathcal{GI}(\xi)}\cap\mathcal{GI}(\xi)_n$ by Lemma \ref{lem4.3}. If $M\in{^\perp\mathcal{GI}}(\xi)\cap\mathcal{GI}(\xi)_n$, then there exists an $\mathbb{E}_\xi$-triangle $\xymatrix@=2em{N\ar[r]&L\ar[r]&M\ar@{-->}[r]&}$ with $N\in\mathcal{GI}(\xi)$ and $L\in\mathcal{I}(\xi)_{n}$ by Theorem \ref{thm3}. It is easy to see that the above $\mathbb{E}_\xi$-triangle splits, and then $M\in \mathcal{I}(\xi)_{n}$.
\end{proof}

%

Following \cite{JSS}, a direct system $[(X_\alpha)_{\alpha<\lambda},(i_{\alpha\beta})_{\alpha<\beta<\lambda}]$ in $\mathcal{B}$ indexed by an ordinal number $\lambda$ is called a \emph{$\lambda$-sequence}
if for each limit ordinal $\gamma<\lambda$, the colimit $\mathrm{colim}_{\alpha<\gamma} X_\alpha$ exists and the colimit morphism
$\mathrm{colim}_{\alpha<\gamma} X_\alpha\rightarrow X_\gamma$ is an isomorphism. If a colimit of the whole direct system exists, we call the
morphism $X_0\rightarrow\mathrm{colim}_{\alpha<\lambda} X_\alpha$ the \emph{composition of the $\lambda$-sequence}.

Given an ordinal $\lambda$,
we say that an increasing sequence $(\lambda_\mu)_{\mu<\nu}$ of strictly smaller ordinals indexed by an ordinal $\nu$
is \emph{cofinal} in $\lambda$ if $\lambda= \mathrm{sup}_{\mu<\nu}\lambda_\mu$. The \emph{cofinality} of $\lambda$ is the smallest ordinal $\nu$
for which there is a cofinal sequence $(\lambda_\mu)_{\mu<\nu}$.

Let $\kappa$ be an infinite regular cardinal and $\mathcal{M}$ a class of morphisms in $\mathcal{B}$. We say that $X$ is \emph{$\kappa$-small}
relative to $\mathcal{M}$ if for all ordinals $\lambda$ of cofinality greater than or equal to $\kappa$ and all $\lambda$-sequences
$Y_0\rightarrow Y_1\rightarrow\cdots \rightarrow Y_\alpha\rightarrow Y_{\alpha+1}\rightarrow\cdots$
such that the composition exists in $\mathcal{B}$ and $Y_\alpha\rightarrow Y_{\alpha+1}$ belongs to $\mathcal{M}$ for each $\alpha +1 <\lambda$, the
canonical morphism
$$\mathrm{colim}_{\alpha<\lambda}\mathrm{Hom}_\mathcal{B}(X,Y_\alpha)\rightarrow\mathrm{Hom}_\mathcal{B}(X,\mathrm{colim}_{\alpha<\lambda}Y_\alpha)$$
is an isomorphism. Finally, $X$ is said to be \emph{small} relative to $\mathcal{M}$ if it is $\kappa$-small relative to $\mathcal{M}$ for
some infinite regular cardinal $\kappa$.

 Consider  the following two axioms:

(Ax1) Arbitrary transfinite compositions of $\xi$-inflations exist and are again $\xi$-inflations;

(Ax2) Every object of $\mathcal{B}$ is small relative to the class of all $\xi$-inflations.

\begin{lem}\label{lem4.5} Assume that $(\mathcal{B}, \mathbb{E}_\xi, \mathfrak{s}_\xi)$
 satisfies {\rm (Ax1)}. Then

$(1)$ The category $\mathcal{B}$ has arbitrary coproducts.

$(2)$ Coproducts of $\mathbb{E}_\xi$-triangles are $\mathbb{E}_\xi$-triangles.
\end{lem}
\begin{proof} (1) Suppose we are given a family $(A_\alpha)_{\alpha\in\Lambda}$ of objects of $\mathcal{B}$ and assume that $\Lambda=\lambda$ is
an ordinal number. We inductively construct $[(X_\alpha)_{\alpha<\lambda},(i_{\alpha\beta})_{\alpha<\beta<\lambda}]$ such that $X_{\alpha+1} = X_\alpha\oplus A_\alpha$ for
$\alpha +1 <\lambda$  and $i_{\alpha,\alpha+1}: X_\alpha\rightarrowtail X_{\alpha+1}$ is the canonical split inflation. Then the coproduct $\coprod_{\alpha<\lambda}A_\alpha$ is the composition
of the $\lambda$-sequence $[(X_\alpha)_{\alpha<\lambda},(i_{\alpha\beta})_{\alpha<\beta<\lambda}]$ by (Ax1).

(2) Suppose $(\xymatrix@=2em{A_\alpha\ar[r]^{\iota_\alpha}&B_\alpha\ar[r]^{\pi_\alpha}&C_\alpha\ar@{-->}[r]&})_{\alpha\in\Lambda}$ is a family of $\mathbb{E}_\xi$-triangles and assume again that $\Lambda=\lambda$ is an
ordinal number. Put $X_\alpha=(\coprod_{\beta<\alpha}B_\beta)\oplus(\coprod_{\beta\geq\alpha}A_\beta)$ for each $\alpha\leq \lambda$. Then
$X_{\alpha+1}=(\coprod_{\beta<\alpha}B_\beta)\oplus B_\alpha\oplus(\coprod_{\beta\geq\alpha+1}A_\beta)$, such that we have an $\mathbb{E}_\xi$-triangle $$\xymatrix@=2em{X_\alpha\ar[r]^{j_{\alpha,\alpha+1}}&X_{\alpha+1}\ar[r]&C_\alpha\ar@{-->}[r]&.}$$
So
$\coprod_{\alpha<\lambda}A_\alpha\rightarrow\coprod_{\alpha<\lambda}B_\alpha$ is a transfinite composition of the $\lambda$-sequence $[(X_\alpha)_{\alpha<\lambda},(j_{\alpha\beta})_{\alpha<\beta<\lambda}]$ by (Ax1). By the universal property of coproducts and  (Ax1),
$$\xymatrix@=2em{\coprod_{\alpha<\lambda}A_\alpha\ar[r]&\coprod_{\alpha<\lambda}B_\alpha\ar[r]&\coprod_{\alpha<\lambda}C_\alpha\ar@{-->}[r]&}$$ is an $\mathbb{E}_\xi$-triangle.
\end{proof}

\begin{lem}\label{lem4.6}  Assume that $(\mathcal{B}, \mathbb{E}_\xi, \mathfrak{s}_\xi)$  satisfies {\rm (Ax1)} and {\rm (Ax2)} and $\mathcal{F}$ is a class of objects in $\mathcal{B}$.
 Suppose that for a fixed object $A\in\mathcal{F}$, $C\in\mathcal{F}^\perp$ if and only
if $\mathbb{E}_\xi(A,C)=0$. Then for any object $M\in\mathcal{B}$ there exists an $\mathbb{E}_\xi$-triangle
$$\xymatrix@=2em{M\ar[r]&C\ar[r]&D\ar@{-->}[r]&,}$$ such that $C\in\mathcal{F}^\perp$ and $D$ is a transfinite composition of some $\lambda$-sequence $[(D_\alpha)_{\alpha<\lambda},(j_{\alpha\beta})_{\alpha<\beta<\lambda}]$ with an $\mathbb{E}_\xi$-triangle $\xymatrix@=2em{D_\alpha\ar[r]^{j_{\alpha,\alpha+1}}&D_{\alpha+1}\ar[r]&A^{(\Lambda_\alpha)}\ar@{-->}[r]&}$ and a set $\Lambda_\alpha$  for all $\alpha<\lambda$. Moreover, if $\mathbb{E}_\xi(D_0,N)=0=\mathbb{E}_\xi(A^{(\Lambda_\alpha)},N)$ for all $\alpha<\lambda$, then $\mathbb{E}_\xi(D,N)=0$ for any $N\in\mathcal{B}$.
\end{lem}
\begin{proof} Let $\xymatrix@=2em{K\ar[r]^{a'}&P\ar[r]^{b'}&A\ar@{-->}[r]&}$ be an $\mathbb{E}_\xi$-triangle with
$P\in\mathcal{P}(\xi)$. Let $M\in\mathcal{B}$ and $c:K^{(\mathcal{B}(K,M))}\rightarrow M$ be the
evaluation map. Set $a=a'^{(\mathcal{B}(K,M))}$ and $b=b'^{(\mathcal{B}(K,M))}$. Hence $a$ is a $\xi$-inflation by Lemma \ref{lem4.5}(2), and $\left[\begin{smallmatrix}a \\ -c\end{smallmatrix}\right]$ is a $\xi$-inflation by \cite[Proposition 1.20]{LN}. Extending $\left[\begin{smallmatrix}a \\ -c\end{smallmatrix}\right]$ to an $\mathbb{E}_\xi$-triangle $\xymatrix@=2em{K^{(\mathcal{B}(K,M))}\ar[r]^{\left[\begin{smallmatrix}a \\ -c\end{smallmatrix}\right]\ \ \ }&P^{(\mathcal{B}(K,M))}\oplus M\ar[r]^{\ \ \ \ \ \ \ \left[\begin{smallmatrix}d & i_{0, 1}\end{smallmatrix}\right]}&M_1\ar@{-->}[r]&.}$
 By the proof of \cite[Proposition 1.20]{LN}, we have a commutative diagram with $\mathbb{E}_\xi$-triangles
$$\xymatrix{&K^{(\mathcal{B}(K,M))}\ar[d]^{\left[\begin{smallmatrix}a \\ -c\end{smallmatrix}\right]}\ar@{=}[r]&K^{(\mathcal{B}(K,M))}\ar[d]^a&\\
M\ar[r]^{\left[\begin{smallmatrix}0 \\ 1\end{smallmatrix}\right]\ \ \ \ \ \ \ \ }\ar@{=}[d]&P^{(\mathcal{B}(K,M))}\oplus M\ar[r]^{\ \ \ \left[\begin{smallmatrix}1 & 0\end{smallmatrix}\right]}\ar[d]^{\left[\begin{smallmatrix}d & i_{0,1}\end{smallmatrix}\right]}&P^{(\mathcal{B}(K,M))}\ar@{-->}[r]\ar[d]^b&\\
M\ar[r]^{i_{0,1}}&M_{1}\ar@{-->}[d]\ar[r]^{f_1\ \ }&A^{(\mathcal{B}(K,M))}\ar@{-->}[d]\ar@{-->}[r]&\\
&&&}$$ it yields the following commutative diagram with $\mathbb{E}_\xi$-triangles
$$\xymatrix{K^{(\mathcal{B}(K,M))}\ar[r]^a\ar[d]^c&P^{(\mathcal{B}(K,M))}\ar[r]^b\ar[d]^d&A^{(\mathcal{B}(K,M))}\ar@{-->}[r]\ar@{=}[d]&\\
M\ar[r]^{i_{0,1}}&M_1\ar[r]^{f_1\ \ }&A^{(\mathcal{B}(K,M))}\ar@{-->}[r]&.}$$
Similarly, one can find an $\mathbb{E}_\xi$-triangle $\xymatrix@=2em{M_1\ar[r]^{i_{1,2}}&M_2\ar[r]^{f_2\ \ \ \ \ }&A^{(\mathcal{B}(K,M_1))}\ar@{-->}[r]&,}$ such that
any morphism $K\rightarrow M_1$ can be extended to a morphism $P\rightarrow M_2$. Repeating this procedure and letting $M_\beta=\mathrm{colim}_{\alpha<\beta}M_\alpha$
whenever $\beta$ is a limit ordinal, where $M_0 = M$. For any ordinal $\lambda$, we can construct a $\lambda$-sequence $[(M_\alpha)_{\alpha<\lambda},(i_{\alpha\beta})_{\alpha<\beta<\lambda}]$ and  a commutative diagram with $\mathbb{E}_\xi$-triangles$$\xymatrix{M\ar[r]^{i_{0,\alpha}}\ar@{=}[d]&M_\alpha\ar[r]\ar[d]^{i_{\alpha,\alpha+1}}&D_\alpha\ar@{-->}[r]\ar[d]^{j_{\alpha,\alpha+1}}&\\
M\ar[r]^{i_{0,\alpha+1}\ \ }&M_{\alpha+1}\ar[r]\ar[d]^{f_{\alpha+1}}&D_{\alpha+1}\ar@{-->}[r]\ar[d]&\\
&A^{(\mathcal{B}(K,M_\alpha))}\ar@{=}[r]\ar@{-->}[d]&A^{(\mathcal{B}(K,M_\alpha))}\ar@{-->}[d]&\\
&&&}$$ for all $\alpha+1<\lambda$. As $K$ is $\kappa$-small relative to $\xi$-inflations for
some infinite regular cardinal $\kappa$ by (Ax2), choose
the least ordinal number $\lambda$ such that $|\lambda|>\kappa$ and set $C=\mathrm{colim}_{\alpha<\lambda}M_\alpha$, $D=\mathrm{colim}_{\alpha<\lambda}D_\alpha$, it yields an $\mathbb{E}_\xi$-triangle
$\xymatrix@=2em{M\ar[r]&C\ar[r]&D\ar@{-->}[r]&}$ by (Ax1). Also any $K\rightarrow C$ factors through $K\rightarrow M_\beta\rightarrow C$ for some $\beta<\lambda$, and
$K\rightarrow M_\beta$ extends to $P\rightarrow M_{\beta+1}$, it means that the sequence  $\mathcal{B}(P, C)\longrightarrow\mathcal{B}(K, C)\longrightarrow 0$ of abelian groups is exact. Applying $\mathcal{B}(-, C)$ to the $\mathbb{E}_\xi$-triangle $\xymatrix@=2em{K\ar[r]&P\ar[r]&A\ar@{-->}[r]&,}$  which implies
 that $\mathbb{E}_\xi(A, C)=0$ as  $\mathbb{E}_\xi(P, C)=0$, then $C\in\mathcal{F}^\perp$ by hypothesis.
Let $N\in\mathcal{B}$ be such that
$\mathbb{E}_\xi(D_0,N)=0=\mathbb{E}_\xi(A^{(\mathcal{B}(K,M_\alpha))},N)$ for $\alpha<\lambda$. It is easy to check that $\mathbb{E}_\xi(D_{\alpha+1},N)=0$ for all $\alpha<\lambda$ by applying $\mathcal{B}(-, N)$ to the $\mathbb{E}_\xi$-triangle $\xymatrix{D_\alpha\ar[r]&D_{\alpha+1}\ar[r]&A^{(\mathcal{B}(K, M_\alpha))}\ar@{-->}[r]&}$ for $\alpha+1<\lambda$.
Consider the
$\mathbb{E}_\xi$-triangle $\xymatrix@=2em{N\ar[r]&E\ar[r]&D\ar@{-->}[r]&.}$ For $\alpha<\lambda$,
one has a commutative diagram with $\mathbb{E}_\xi$-triangles$$\xymatrix{N\ar[r]^{g_\alpha}\ar@{=}[d]&E_\alpha\ar[r]^{h_\alpha}\ar[d]&D_\alpha\ar@{-->}[r]\ar[d]&\\
N\ar[r]^g&E\ar[r]^h&D\ar@{-->}[r]&.}$$
As the upper row is split, there is $\Phi_\alpha:E_\alpha\rightarrow N$ such that $\Phi_\alpha g_\alpha=\mathrm{id}_N$,
it follows from the universal property of colimits that there is $\Phi=\mathrm{colim}_{\alpha<\lambda}\Phi_\alpha:E\rightarrow N$ such that $\Phi g=\mathrm{id}_N$. So the lower row is split, and hence $\mathbb{E}_\xi(D,N)=0$.
\end{proof}

Let $\mathcal{X}$ and $\mathcal{Y}$ be subcategories of $(\mathcal{B}, \mathbb{E}_\xi, \mathfrak{s}_\xi)$. Set
$$\mathrm{Cone}(\mathcal{X},\mathcal{Y}):=\{C\in\mathcal{B}\hspace{0.03cm}|\hspace{0.03cm}\exists\ \mathbb{E}_\xi\textrm{-triangle}\ \xymatrix@=2em{X\ar[r]&Y\ar[r]&C\ar@{-->}[r]&}\ \textrm{with}\ X\in\mathcal{X}\ \textrm{and}\ Y\in\mathcal{Y}\},$$
$$\mathrm{CoCone}(\mathcal{X},\mathcal{Y}):=\{D\in\mathcal{B}\hspace{0.03cm}|\hspace{0.03cm}\exists\ \mathbb{E}_\xi\textrm{-triangle}\ \xymatrix@=2em{D\ar[r]&X\ar[r]&Y\ar@{-->}[r]&}\ \textrm{with}\ X\in\mathcal{X}\ \textrm{and}\ Y\in\mathcal{Y}\},$$
and let $\Sigma\mathcal{X}=\mathrm{Cone}(\mathcal{X},\mathcal{I}(\xi))$ and $\Sigma^n\mathcal{X}=\mathrm{Cone}(\Sigma^{n-1}\mathcal{X},\mathcal{I}(\xi))$, $\Omega\mathcal{X}=\mathrm{CoCone}(\mathcal{P}(\xi),\mathcal{X})$ and $\Omega^n\mathcal{X}=\mathrm{Cone}(\mathcal{P}(\xi),\Omega^{n-1}\mathcal{X})$ for $n\geq2$.

\begin{prop}\label{lem4.7}  Assume that $(\mathcal{B}, \mathbb{E}_\xi, \mathfrak{s}_\xi)$  satisfies {\rm (Ax1)} and {\rm (Ax2)}.

$(1)$ If the pair $({^\bot}\mathcal{I}(\xi),\mathcal{I}(\xi))$ is cogenerated by a set $\mathcal{S}$, then the pair $({^\bot}\mathcal{I}(\xi)_n,\mathcal{I}(\xi)_n)$ is cogenerated by the set $\Omega^n\mathcal{S}$. Moreover, $({^\bot}\mathcal{I}(\xi)_n,\mathcal{I}(\xi)_n)$ is a complete cotorsion pair.

$(2)$ If the cotorsion pair $({^\bot}\mathcal{GI}(\xi),\mathcal{GI}(\xi))$ is cogenerated by a set $\mathcal{S}'$, then it is complete.
\end{prop}
\begin{proof} (1) Let $X\in\mathcal{I}(\xi)_n$. Then $\mathbb{E}_\xi(\Omega^nS,X)\cong\mathbb{E}^{n+1}_\xi(S,X)=0$ for all $S\in\mathcal{S}$. If $\mathbb{E}_\xi(\Omega^nS,X)=0$ for all $S\in\mathcal{S}$, then $\mathbb{E}_\xi(S,\Sigma^nX)\cong\mathbb{E}^{n+1}_\xi(S,X)\cong\mathbb{E}_\xi(\Omega^nS,X)=0$ for all $S\in\mathcal{S}$, it implies that $\Sigma^nX\in\mathcal{I}(\xi)$ and $X\in\mathcal{I}(\xi)_n$. Hence we have $\mathcal{I}(\xi)_n=(\Omega^n\mathcal{S})^\perp$, which implies that $(^\perp\mathcal{I}(\xi)_n, \mathcal{I}(\xi)_n)$ is a cotorsion pair. For any $L\in\mathcal{B}$, there exists an $\mathbb{E}_\xi$-triangle $\xymatrix{M\ar[r]&P\ar[r]&L\ar@{-->}[r]&}$ with $P\in\mathcal{P}(\xi)$ since $\mathcal{B}$ has enough $\xi$-projectives. Set $A=\coprod_{S\in\mathcal{S}}\Omega^nS$. By Lemma \ref{lem4.6}, there is an $\mathbb{E}_\xi$-triangle
$\xymatrix@=2em{M\ar[r]&C\ar[r]&D\ar@{-->}[r]&}$ such that $C\in\mathcal{I}(\xi)_n$ and $D$ is a transfinite composition of some $\lambda$-sequence $[(D_\alpha)_{\alpha<\lambda},(j_{\alpha\beta})_{\alpha<\beta<\lambda}]$ with an $\mathbb{E}_\xi$-triangle $\xymatrix@=2em{D_\alpha\ar[r]^{j_{\alpha,\alpha+1}}&D_{\alpha+1}\ar[r]&A^{(\Lambda_\alpha)}\ar@{-->}[r]&}$ for all $\alpha<\lambda$. For any $X\in\mathcal{I}(\xi)_n$ and $\alpha<\lambda$,
as
$\mathbb{E}_\xi(D_0,X)=0$ and $\mathbb{E}_\xi(A^{(\Lambda_\alpha)},X)=0$, it follows from Lemma \ref{lem4.6} that $\mathbb{E}_\xi(D,X)=0$, so $D\in{^\bot}\mathcal{I}(\xi)_n$.
It follows from the dual version of Lemma \ref{lem1} that there is a commutative diagram with $\mathbb{E}_\xi$-triangles
 $$\xymatrix{M\ar[r]\ar[d]&P\ar[r]\ar[d]&L\ar@{-->}[r]\ar@{=}[d]&\\
C\ar[d]\ar[r]&E\ar[r]\ar[d]&L\ar@{-->}[r]&\\
D\ar@{=}[r]\ar@{-->}[d]&D\ar@{-->}[d]&&\\
&&&}$$
with $P\in\mathcal{P}(\xi)$, $C\in\mathcal{I}(\xi)_n$ and $D\in{^\bot}\mathcal{I}(\xi)_n$. Then $E\in{^\bot}\mathcal{I}(\xi)_n$ as $P, D\in {^\bot}\mathcal{I}(\xi)_n$, so there is an $\mathbb{E}_\xi$-triangle $\xymatrix{C\ar[r]&E\ar[r]&L\ar@{-->}[r]&}$ with $E\in{^\bot}\mathcal{I}(\xi)_n$ and $C\in\mathcal{I}(\xi)_n$. Then $({^\bot}\mathcal{I}(\xi)_n,\mathcal{I}(\xi)_n)$ is a complete cotorsion pair by Lemma \ref{lemma2.18}.

(2) By Lemma \ref{lem4.6}, there is an $\mathbb{E}_\xi$-triangle
$\xymatrix@=2em{M\ar[r]&C\ar[r]&D\ar@{-->}[r]&}$ such that $C\in\mathcal{GI}(\xi)$ and $D\in{^\perp}\mathcal{GI}(\xi)$.
Thus the cotorsion pair $({^\bot}\mathcal{GI}(\xi),\mathcal{GI}(\xi))$ is complete by Lemma \ref{lemma2.18}.
\end{proof}

We are ready to prove Theorem \ref{thm4} in the introduction.

{\bf Proof of Theorem \ref{thm4}.} Set $(\mathcal{C}, \mathcal{W},\mathcal{F})=(\mathcal{B}, {^\perp\mathcal{GI}}(\xi),\mathcal{GI}(\xi))$. Then $^\perp\mathcal{GI}(\xi)$ is thick in $(\mathcal{B}, \mathbb{E}_\xi, \mathfrak{s}_\xi)$ and
$\mathcal{GI}(\xi)\cap{^\perp\mathcal{GI}}(\xi)=\mathcal{I}(\xi)$ by Lemma \ref{lem4.2}. Hence the complete cotorsion pair $(\mathcal{C}\cap\mathcal{W}, \mathcal{F})=(^\perp\mathcal{GI}(\xi), \mathcal{GI}(\xi))$ is hereditary by Lemma \ref{lem2.18}. It is clear that $(\mathcal{C}, \mathcal{W}\cap\mathcal{F})=(\mathcal{B}, \mathcal{I}(\xi))$ is a complete and hereditary cotorsion pair. Hence $(\mathcal{C}, \mathcal{W},\mathcal{F})=(\mathcal{B}, {^\perp\mathcal{GI}}(\xi),\mathcal{GI}(\xi))$ is a hereditary Hovey triple in $\mathcal{B}$. By Proposition \ref{lem4.7}(1), we obtain that $(^\perp\mathcal{I}(\xi)_n, \mathcal{I}(\xi)_n)$ is a complete cotorsion pair in $(\mathcal{B}, \mathbb{E}_\xi, \mathfrak{s}_\xi)$. Thus the statements follow from Theorem \ref{thm2} and Corollary \ref{cor4.7} as $(^\perp(\mathcal{W}\cap\mathcal{F}), (\mathcal{W}\cap\mathcal{F}))=(^\perp\mathcal{I}(\xi)_n, \mathcal{I}(\xi)_n)$.\hfill$\Box$
\vspace{2mm}

Note that $(^{\perp}\mathcal{GI}(R),\mathcal{GI}(R))$ is a complete cotorsion pair for any ring $R$ by \cite{SS}. A direct consequence of Theorem \ref{thm4} yields the following result.

\begin{cor}\label{cor:4.11}{\rm  (\cite[Corollary 3.9]{El Maaouy2})} Let $R$ be a ring, and $n$ any non-negative integer.

$(1)$ The pair $(^\perp\mathcal{I}(R)\cap{^\perp\mathcal{GI}}(R), \mathcal{GI}(R)_n)$ is a complete hereditary cotorsion pair in $\mathrm{Mod}R$ with kernel $\mathcal{I}(R)_n\cap{^\perp\mathcal{I}}(R)_n$.

$(2)$ The triple $(^\perp\mathcal{I}(R)_n, {^\perp}\mathcal{GI}(R), \mathcal{GI}(R)_n)$ is a hereditary Hovey triple in $\mathrm{Mod}R$; $^{\perp}\mathcal{I}(R)_n\cap\mathcal{GI}(R)_n$ is a Frobenius category such that $^{\perp}\mathcal{I}(R)_n\cap\mathcal{I}(R)_n$ is its class of projective-injective objects; and the homotopy category is the stable category $\underline{^{\perp}\mathcal{I}(R)_n\cap\mathcal{GI}(R)_n}$, which is triangle equivalent to $\underline{\mathcal{GI}(R)}$.
\end{cor}

Let $R$ be a ring. We denote by $\mathrm{C}(R)$ the category whose class of objects consists of all
cochain complexes $X$ of $R$-modules,
$$\cdots\longrightarrow X^{s-1}\stackrel{d^{s-1}_X}\longrightarrow X^s\stackrel{d^{s}_X}\longrightarrow X^{s+1}\longrightarrow\cdots,$$such that $d^{s}_Xd^{s-1}_X=0$ for all $s\in\mathbb{Z}$.
Denotes by $\mathrm{D}(R)$ the derived category of cochain complexes of $R$-modules. For $s\in\mathbb{Z}$, denote $$\mathrm{Z}^s(X)=\mathrm{ker}d_X^{s},\ \mathrm{B}^{s}(X)=\mathrm{im}d_X^{s-1},\ \mathrm{C}^{s}(X)=\mathrm{coker}d_X^{s-1},\ \mathrm{H}^s(X)=\mathrm{Z}^s(X)/\mathrm{B}^s(X).$$
Following \cite{O}, an exact triangle $X\longrightarrow Y\longrightarrow Z\longrightarrow X[1]$ in $\mathrm{D}(R)$ is called a \emph{cohomologically ghost triangle} if $0\longrightarrow\mathrm{H}^s(X)\longrightarrow\mathrm{H}^s(Y)\longrightarrow\mathrm{H}^s(Z)\longrightarrow0$ is exact for all $s\in\mathbb{Z}$. Given a class $\mathcal{X}$ of $R$-modules, a complex $X$ is called a \emph{CE $\mathcal{X}$
complex} if $X^s$, $\mathrm{Z}^s(X)$, $\mathrm{B}^s(X)$ and $\mathrm{H}^s(X)$ are all in $\mathcal{X}$ for all $s\in\mathbb{Z}$.
In particular, if $\mathcal{X}$ is the class $\mathcal{P}(R)$ (resp.  $\mathcal{I}(R)$, $\mathcal{GI}(R)$) of
projective (resp., injective, Gorenstein injective) $R$-modules, then a CE $\mathcal{X}$ complex is just a CE projective $\mathrm{CE}(\mathcal{P}(R))$ complex (resp., CE injective $\mathrm{CE}(\mathcal{I}(R))$ complex, CE Gorenstein injective $\mathrm{CE}(\mathcal{GI}(R))$ complex), see \cite{E,YL}. By \cite[Proposition 6.3]{E}, the functor $\mathrm{Hom}(-,-)$ on
$\mathrm{C}(R)\times\mathrm{C}(R)$ is right balanced by $\mathrm{CE}(\mathcal{P}(R))\times\mathrm{CE}(\mathcal{I}(R))$. So one can compute derived functors of $\mathrm{Hom}(-,-)$ using either of the two resolutions, and denote these functors by $\overline{\mathrm{Ext}}^n(-,-)$.

\begin{cor}\label{cor:4.12} Let $R$ be a ring and $\mathcal{B}=\mathrm{D}(R)$ with $\mathbb{E}_\xi$ the class of cohomological ghost triangles in $\mathrm{D}(R)$.

$(1)$ The pair $(^\perp\mathcal{I}(\xi)\cap{^\perp\mathcal{GI}}(\xi), \mathcal{GI}(\xi)_n)$ is a complete hereditary cotorsion pair in $(\mathcal{B},\mathbb{E}_\xi,\mathfrak{s}_\xi)$ with core $\mathcal{I}(\xi)_n\cap{^\perp\mathcal{I}}(\xi)_n$.

$(2)$ The triple $(^\perp\mathcal{I}(\xi)_n, ^\perp\mathcal{GI}(\xi), \mathcal{GI}(\xi)_n)$ is a hereditary Hovey triple in $(\mathcal{B},\mathbb{E}_\xi,\mathfrak{s}_\xi)$; $^{\perp}\mathcal{I}(\xi)_n\cap\mathcal{GI}(\xi)_n$ is a Frobenius extriangulated category such that $^{\perp}\mathcal{I}(\xi)_n\cap\mathcal{I}(\xi)_n$ is its class of projective-injective objects; and the homotopy category is the stable category $\underline{^{\perp}\mathcal{I}(\xi)_n\cap\mathcal{GI}(\xi)_n}$, which is triangle equivalent to $\underline{\mathcal{GI}(\xi)}$.
\end{cor}
\begin{proof} By \cite[Proposition 3.10]{Y}, one has $$\mathcal{I}(\xi)_n=
\{X\in\mathrm{D}(R)\hspace{0.03cm}|\hspace{0.03cm}\mathrm{Z}^{s}(I)\in\mathcal{I}(R)_n\ \textrm{for\ some\ semi-injective resolution}\ X\stackrel{\simeq}\rightarrow I\ \textrm{and\ all}\ s\in\mathbb{Z}\}.$$
By \cite[Theorem 9.4]{E}, $(\mathrm{CE}({^\perp}\mathcal{I}(R)_n),\mathrm{CE}(\mathcal{I}(R)_n))$ is a cotorsion pair in $\mathrm{C}(R)$ relative to $\overline{\mathrm{Ext}}^1(-,-)$, and there exists  $S\in\mathrm{CE}({^\perp}\mathcal{I}(R)_n)$ such that $\mathrm{CE}(\mathcal{I}(R)_n)=\{X\in\mathrm{C}(R)\hspace{0.03cm}|\hspace{0.03cm}
\overline{\mathrm{Ext}}^1(S,X)=0\}$. Assume that $\mathbb{E}_\xi(S,Y)=0$ for $Y\in\mathrm{D}(R)$. Then  $\overline{\mathrm{Ext}}^1(S,I)=0$ by \cite[Lemma 2.1]{Y}, where $Y\stackrel{\simeq}\rightarrow I$ is a semi-injective resolution, so $I\in\mathrm{CE}(\mathcal{I}(R)_n)$, and therefore $Y\in\mathcal{I}(\xi)_n$ by \cite[Lemma 2.3]{Y}.
 Thus  $({^\bot}\mathcal{I}(\xi)_n,\mathcal{I}(\xi)_n)$ is a cotorsion pair in $\mathcal{B}$ by \cite[Theorem 2.5]{Y}, that is cogenerated by a set.
Also  by \cite[Proposition 3.3]{Y}, one has
$$\mathcal{GI}(\xi)=
\{X\in\mathrm{D}(R)\hspace{0.03cm}|\hspace{0.03cm}\mathrm{Z}^{s}(I)\in\mathcal{GI}(R)\ \textrm{for\ some\ semi-injective resolution}\ X\stackrel{\simeq}\rightarrow I\ \textrm{and\ all}\ s\in\mathbb{Z}\}.$$
By \cite[Theorem 9.4]{E} and \cite[Theorem 5.6]{SS}, $(\mathrm{CE}({^\perp}\mathcal{GI}(R)),\mathrm{CE}(\mathcal{GI}(R)))$ is a cotorsion pair in $\mathrm{C}(R)$ relative to $\overline{\mathrm{Ext}}^1(-,-)$, which is cogenerated by a set. By analogy with the preceding proof,
 the cotorsion pair $({^\perp}\mathcal{GI}(\xi), \mathcal{GI}(\xi))$ is cogenerated by a set. The statements are a direct consequence of Theorem \ref{thm4}.
\end{proof}


\section{\bf Model structures on $\mathcal{GI}(\xi)_n$}\label{section5}
The main aim of this section is to construct model structures on the extriangulated category induced by objects whose $\xi$-$\mathcal{G}$injective dimension is bounded by a non-negative integer, and to prove Theorem \ref{thm4.15'}.
Let $(\mathcal{B}, \mathbb{E}_\xi, \mathfrak{s}_\xi)$ be an extriangulated category, and $\mathcal{D}$ a full subcategory of $\mathcal{B}$ which is closed under extensions. It follows from \cite[Remark 2.18]{NP} that there is an extriangulated category $(\mathcal{D}, \mathbb{E}_{\mathcal{D}}, \mathfrak{s}_{\mathcal{D}})$, in which $\mathbb{E}_{\mathcal{D}}$ is the restriction of $\mathbb{E}_\xi$ onto $\mathcal{D}^{op}\times \mathcal{D}$ and  $\mathfrak{s}_{\mathcal{D}}$ is the restriction of $\mathfrak{s}_\xi$ onto $\mathcal{D}$. In particular, $(\mathcal{GI}(\xi)_n, \mathbb{E}_{\mathcal{GI}(\xi)_n}, \mathfrak{s}_{\mathcal{GI}(\xi)_n})$ is an extriangulated category with enough projectives and enough injectives by Lemma \ref{lem4.4}.

Let $(\mathcal{D}, \mathbb{E}_{\mathcal{D}}, \mathfrak{s}_{\mathcal{D}})$ be an extriangulated category as above, and $\mathcal{Z}$ a full subcategory of $\mathcal{D}$. Set

\medskip\noindent~~~~~$^{\perp_\mathcal{D}}\mathcal{Z}=\{X\in\mathcal{D}~|~\mathbb{E}_\mathcal{D}(X, Z)=0, \forall~Z\in\mathcal{Z}\}$ and $\mathcal{Z}^{\perp_\mathcal{D}}=\{X\in\mathcal{D}~|~\mathbb{E}_\mathcal{D}(Z, X)=0, \forall~Z\in\mathcal{Z}\}$.

\medskip

\begin{lem}\label{lemma4.8} Let $(\mathcal{D}, \mathbb{E}_{\mathcal{D}}, \mathfrak{s}_{\mathcal{D}})$ be an extriangulated category as above, and $\mathcal{Z}$ a full subcategory of $\mathcal{D}$. Then $^{\perp_\mathcal{D}}\mathcal{Z}={^\perp\mathcal{Z}}\cap\mathcal{D}$ and $\mathcal{Z}^{\perp_\mathcal{D}}=\mathcal{Z}^\perp\cap\mathcal{D}$.
\end{lem}
\begin{proof} First we prove that $^{\perp_\mathcal{D}}\mathcal{Z}={^\perp\mathcal{Z}}\cap\mathcal{D}$. It is easy to prove that $^{\perp_\mathcal{D}}\mathcal{Z}\subseteq {^\perp\mathcal{Z}}\cap\mathcal{D}$. On the other hand, if $X\in{^\perp\mathcal{Z}}\cap\mathcal{D}$, then $X\in{^\perp\mathcal{Z}}$ and $X\in\mathcal{D}$. For any $Z\in\mathcal{Z}$, one has $\mathbb{E}_\mathcal{D}(X, Z)=\mathbb{E}_\xi(X, Z)=0$ since $\mathcal{Z}\subseteq \mathcal{D}$. Hence $X\in{^{\perp_\mathcal{D}}}\mathcal{Z}$. Then $^{\perp_\mathcal{D}}\mathcal{Z}={^\perp\mathcal{Z}}\cap\mathcal{D}$.

Dually, we can prove that $\mathcal{Z}^{\perp_\mathcal{D}}=\mathcal{Z}^\perp\cap\mathcal{D}$.
\end{proof}

\begin{thm}\label{thm4.9}Let $\mathcal{D}$ be a full subcategory of $(\mathcal{B}, \mathbb{E}_\xi, \mathfrak{s}_\xi)$ which is closed under extensions and closed under cones of inflations. Suppose that $(\mathcal{X}, \mathcal{Y})$ is a complete cotorsion pair in $(\mathcal{B}, \mathbb{E}_\xi, \mathfrak{s}_\xi)$ and $\mathcal{Y}\subseteq\mathcal{D}$. Then

$(1)$  $(\mathcal{X}\cap\mathcal{D}, \mathcal{Y})$ is a complete cotorsion pair in the extriangulated category $(\mathcal{D}, \mathbb{E}_{\mathcal{D}}, \mathfrak{s}_{\mathcal{D}})$.

$(2)$ If  $(\mathcal{X}, \mathcal{Y})$ is a hereditary cotorsion pair in $(\mathcal{B}, \mathbb{E}_\xi, \mathfrak{s}_\xi)$, then $(\mathcal{X}\cap\mathcal{D}, \mathcal{Y})$ is a hereditary cotorsion pair in the extriangulated category $(\mathcal{D}, \mathbb{E}_{\mathcal{D}}, \mathfrak{s}_{\mathcal{D}})$.
\end{thm}
\begin{proof} (1) It is clear that $(\mathcal{D}, \mathbb{E}_{\mathcal{D}}, \mathfrak{s}_{\mathcal{D}})$ is an extriangulated category  because $\mathcal{D}$ is a full subcategory of $(\mathcal{B}, \mathbb{E}_\xi, \mathfrak{s}_\xi)$ which is closed under extensions. To prove that $(\mathcal{X}\cap\mathcal{D}, \mathcal{Y})$ is a cotorsion pair in $(\mathcal{D}, \mathbb{E}_{\mathcal{D}}, \mathfrak{s}_{\mathcal{D}})$, it suffices to prove that $^\perp\mathcal{Y}\cap\mathcal{D}=\mathcal{X}\cap\mathcal{D}$ and $(\mathcal{X}\cap\mathcal{D})^\perp\cap\mathcal{D}=\mathcal{Y}$ by Lemma \ref{lemma4.8}. Since $(\mathcal{X}, \mathcal{Y})$ is a complete cotorsion pair in $(\mathcal{B}, \mathbb{E}_\xi, \mathfrak{s}_\xi)$, we have $^\perp\mathcal{Y}\cap\mathcal{D}=\mathcal{X}\cap\mathcal{D}$. It is clear that $\mathcal{Y}=\mathcal{X}^\perp\subseteq{(\mathcal{X}\cap\mathcal{D})}^\perp\cap\mathcal{D}$. If $M\in{(\mathcal{X}\cap\mathcal{D})^\perp}\cap\mathcal{D}$, then $M\in{(\mathcal{X}\cap\mathcal{D})^\perp}$ and $M\in\mathcal{D}$. Since $(\mathcal{X}, \mathcal{Y})$ is a complete cotorsion pair in $(\mathcal{B}, \mathbb{E}_\xi, \mathfrak{s}_\xi)$, there is an $\mathbb{E}_\xi$-triangle $\xymatrix{M\ar[r]&Y\ar[r]&X\ar@{-->}[r]&}$ with $X\in\mathcal{X}$ and $Y\in\mathcal{Y}$. Since $Y\in\mathcal{Y}\subseteq\mathcal{D}$ and $\mathcal{D}$ is closed under cones of inflations, we have $X\in\mathcal{X}\cap\mathcal{D}$. It is easy to check that the above $\mathbb{E}_\xi$-triangle splits, and $M\in\mathcal{Y}$. Hence $(\mathcal{X}\cap\mathcal{D})^\perp\cap\mathcal{D}=\mathcal{Y}$. Therefore, $(\mathcal{X}\cap\mathcal{D}, \mathcal{Y})$ is a cotorsion pair in $(\mathcal{D}, \mathbb{E}_{\mathcal{D}}, \mathfrak{s}_{\mathcal{D}})$.

For $M\in\mathcal{D}$, there are $\mathbb{E}_\xi$-triangles $\xymatrix@=2em{Y\ar[r]&X\ar[r]&M\ar@{-->}[r]&}$ and  $\xymatrix@=2em{M\ar[r]&Y'\ar[r]&X'\ar@{-->}[r]&}$ with $X, X'\in\mathcal{X}$ and $Y, Y'\in\mathcal{Y}$ since $(\mathcal{X}, \mathcal{Y})$ is complete in $(\mathcal{B}, \mathbb{E}_\xi, \mathfrak{s}_\xi)$. Note that $\mathcal{Y}\subseteq\mathcal{D}$. Then $X, X'\in\mathcal{D}$ because $\mathcal{D}$ is closed  under extensions and closed under cones of inflations. Therefore, there exist $\mathbb{E}_\xi$-triangles $\xymatrix@=2em{Y\ar[r]&X\ar[r]&M\ar@{-->}[r]&}$ and  $\xymatrix@=2em{M\ar[r]&Y'\ar[r]&X'\ar@{-->}[r]&}$ with $X, X'\in\mathcal{X}\cap\mathcal{D}$ and $Y, Y'\in\mathcal{Y}$. Thus $(\mathcal{X}\cap\mathcal{D}, \mathcal{Y})$ is a complete cotorsion pair in $(\mathcal{D}, \mathbb{E}_{\mathcal{D}}, \mathfrak{s}_{\mathcal{D}})$.

(2) Since $(\mathcal{X}, \mathcal{Y})$ is a hereditary cotorsion pair in $(\mathcal{B}, \mathbb{E}_\xi, \mathfrak{s}_\xi)$, $\mathcal{Y}$ is closed under cones of inflations in $(\mathcal{B}, \mathbb{E}_\xi, \mathfrak{s}_\xi)$ by Lemma \ref{lem2.18}. Hence $\mathcal{Y}$ is closed under cones of inflations in $\mathcal{D}$, and  $(\mathcal{X}\cap\mathcal{D}, \mathcal{Y})$ is a hereditary cotorsion pair in the extriangulated category $(\mathcal{D}, \mathbb{E}_{\mathcal{D}}, \mathfrak{s}_{\mathcal{D}})$ by Lemma \ref{lem2.18}.\end{proof}

\begin{cor}\label{cor4.11}Let 
$m$ and $n$ be non-negative integers with $m\leqslant n$.
If $(^\perp\mathcal{I}(\xi)_m, \mathcal{I}(\xi)_m)$ is a complete and hereditary cotorsion pair in  $(\mathcal{B}, \mathbb{E}_\xi, \mathfrak{s}_\xi)$, then $(^\perp\mathcal{I}(\xi)_m\cap\mathcal{GI}(\xi)_n, \mathcal{I}(\xi)_m)$ is a complete and hereditary cotorsion pair in the extriangulated category $(\mathcal{GI}(\xi)_n, \mathbb{E}_{\mathcal{GI}(\xi)_n}, \mathfrak{s}_{\mathcal{GI}(\xi)_n})$.
\end{cor}
\begin{proof} It follows from Lemma \ref{lem4.4} that $\mathcal{GI}(\xi)_n$ is closed under extensions and closed under cones of inflations in $(\mathcal{B}, \mathbb{E}_\xi, \mathfrak{s}_\xi)$, hence $(\mathcal{GI}(\xi)_n, \mathbb{E}_{\mathcal{GI}(\xi)_n}, \mathfrak{s}_{\mathcal{GI}(\xi)_n})$ is an extriangulated category. Since $m\leqslant n$, $\mathcal{I}(\xi)_m\subseteq\mathcal{GI}(\xi)_m\subseteq \mathcal{GI}(\xi)_n$, and $(^\perp\mathcal{I}(\xi)_m\cap\mathcal{GI}(\xi)_n, \mathcal{I}(\xi)_m)$ is a complete and hereditary cotorsion pair in the extriangulated category $(\mathcal{GI}(\xi)_n, \mathbb{E}_{\mathcal{GI}(\xi)_n}, \mathfrak{s}_{\mathcal{GI}(\xi)_n})$ by Theorem \ref{thm4.9}.
\end{proof}
\begin{lem}\label{lemma4.11}Let 
$m$ and $n$ be non-negative integers with $m\leqslant n$.
If $(^\perp\mathcal{I}(\xi)_m, \mathcal{I}(\xi)_m)$ is a complete and hereditary cotorsion pair in  $(\mathcal{B}, \mathbb{E}_\xi, \mathfrak{s}_\xi)$, then $(^\perp\mathcal{I}(\xi)_m\cap\mathcal{I}(\xi)_n)^\perp\cap\mathcal{I}(\xi)_n=\mathcal{I}(\xi)_m$.
\end{lem}
\begin{proof} It is clear that $\mathcal{I}(\xi)_m\subseteq{(^\perp\mathcal{I}(\xi)_m\cap\mathcal{I}(\xi)_n)^\perp}\cap\mathcal{I}(\xi)_n$. If $M\in{(^\perp\mathcal{I}(\xi)_m\cap\mathcal{I}(\xi)_n)^\perp}\cap\mathcal{I}(\xi)_n$, then there exists an $\mathbb{E}_\xi$-triangle $\xymatrix@=2em{M\ar[r]&H\ar[r]&N\ar@{-->}[r]&}$ with $H\in\mathcal{I}(\xi)_m$ and $N\in^\perp\mathcal{I}(\xi)_m$ since $(^\perp\mathcal{I}(\xi)_m, \mathcal{I}(\xi)_m)$ is a complete cotorsion pair in $(\mathcal{B}, \mathbb{E}_\xi, \mathfrak{s}_\xi)$. Hence $N\in\mathcal{I}(\xi)_n$ because $\mathcal{I}(\xi)_n$ is closed under cones of inflations. Then $N\in{^\perp\mathcal{I}(\xi)_m\cap\mathcal{I}(\xi)_n}$, and $\mathbb{E}_\xi(N, M)=0$. So the above $\mathbb{E}_\xi$-triangle splits, which implies $M\in\mathcal{I}(\xi)_m$.
\end{proof}

\begin{thm}\label{thm4.15} Let 
$m$ and $n$ be non-negative integers with $m\leqslant n$.
If $(^\perp\mathcal{I}(\xi)_m, \mathcal{I}(\xi)_m)$ is a complete and hereditary cotorsion pair in  $(\mathcal{B}, \mathbb{E}_\xi, \mathfrak{s}_\xi)$, then the following hold.

$(1)$ The pair $(^\perp\mathcal{I}(\xi)_m\cap\mathcal{I}(\xi)_n, \mathcal{GI}(\xi)_m)$ is a complete and hereditary cotorsion pair in the extriangulated category $(\mathcal{GI}(\xi)_n, \mathbb{E}_{\mathcal{GI}(\xi)_n}, \mathfrak{s}_{\mathcal{GI}(\xi)_n})$ with core $^\perp\mathcal{I}(\xi)_m\cap\mathcal{I}(\xi)_m$.

$(2)$ The triple $(^\perp\mathcal{I}(\xi)_m\cap\mathcal{GI}(\xi)_n, \mathcal{I}(\xi)_n, \mathcal{GI}(\xi)_m)$ is a hereditary Hovey triple in the extriangulated category $(\mathcal{GI}(\xi)_n, \mathbb{E}_{\mathcal{GI}(\xi)_n}, \mathfrak{s}_{\mathcal{GI}(\xi)_n})$. The corresponding homotopy category is
$$\underline{^\perp\mathcal{I}(\xi)_m\cap\mathcal{GI}(\xi)_m}\cong \underline{\mathcal{GI}(\xi)}.$$
\end{thm}
\begin{proof} (1) It is easy to see that $(\mathcal{GI}(\xi)_n, \mathbb{E}_{\mathcal{GI}(\xi)_n}, \mathfrak{s}_{\mathcal{GI}(\xi)_n})$ is an extriangulated category with enough projectives and enough injectives. To prove that $(^\perp\mathcal{I}(\xi)_m\cap\mathcal{I}(\xi)_n, \mathcal{GI}(\xi)_m)$ is a cotorsion pair in $(\mathcal{GI}(\xi)_n, \mathbb{E}_{\mathcal{GI}(\xi)_n}, \mathfrak{s}_{\mathcal{GI}(\xi)_n})$, it suffices to show that $^\perp\mathcal{GI}(\xi)_m\cap\mathcal{GI}(\xi)_n={^\perp\mathcal{I}}(\xi)_m\cap\mathcal{I}(\xi)_n$ and $(^\perp\mathcal{I}(\xi)_m\cap\mathcal{I}(\xi)_n)^\perp\cap\mathcal{GI}(\xi)_n=\mathcal{GI}(\xi)_m$ by Lemma \ref{lemma4.8}.

Since $\mathcal{I}(\xi)_n\subseteq {^\perp\mathcal{GI}}(\xi)$ by Lemma \ref{lem4.3}, one has $^\perp\mathcal{I}(\xi)_m\cap\mathcal{I}(\xi)_n\subseteq {^\perp\mathcal{I}}(\xi)_m\cap{^\perp\mathcal{GI}}(\xi)$. It follows from Theorem \ref{thm3}(4) that

\qquad\qquad$\mathbb{E}_\xi(^\perp\mathcal{I}(\xi)_m\cap\mathcal{I}(\xi)_n, \mathcal{GI}(\xi)_m)=0,$ \hfill $(*)$\qquad

 \noindent which implies that $^\perp\mathcal{I}(\xi)_m\cap\mathcal{I}(\xi)_n\subseteq{^\perp\mathcal{GI}}(\xi)_m\cap\mathcal{GI}(\xi)_n$ and $\mathcal{GI}(\xi)_m\subseteq (^\perp\mathcal{I}(\xi)_m\cap\mathcal{I}(\xi)_n)^\perp\cap\mathcal{GI}(\xi)_n$.

Since $^\perp\mathcal{GI}(\xi)_m\subseteq {^\perp\mathcal{I}}(\xi)_m\cap{^\perp\mathcal{GI}}(\xi)$ by Theorem \ref{thm3}, and $^\perp\mathcal{GI}(\xi)\cap\mathcal{GI}(\xi)_n=\mathcal{I}(\xi)_n$ by Corollary \ref{cor4.7}, one has $^\perp\mathcal{GI}(\xi)_m\cap\mathcal{GI}(\xi)_n\subseteq {^\perp\mathcal{I}}(\xi)_m\cap{^\perp\mathcal{GI}}(\xi)\cap\mathcal{GI}(\xi)_n
={^\perp\mathcal{I}}(\xi)_m\cap\mathcal{I}(\xi)_n$. Hence $^\perp\mathcal{GI}(\xi)_m\cap\mathcal{GI}(\xi)_n={^\perp\mathcal{I}}(\xi)_m\cap\mathcal{I}(\xi)_n$.

It remains to prove $(^\perp\mathcal{I}(\xi)_m\cap\mathcal{I}(\xi)_n)^\perp\cap\mathcal{GI}(\xi)_n\subseteq\mathcal{GI}(\xi)_m$. We may assume that $n\geqslant 1$. Let $M\in(^\perp\mathcal{I}(\xi)_m\cap\mathcal{I}(\xi)_n)^\perp\cap\mathcal{GI}(\xi)_n$. Then  there is an $\mathbb{E}_\xi$-triangle $\xymatrix@=2em{M\ar[r]&G\ar[r]&K\ar@{-->}[r]&}$ with $K\in\mathcal{I}(\xi)_{n-1}$ and $G\in\mathcal{GI}(\xi)$ by Theorem \ref{thm3} as $M\in\mathcal{GI}(\xi)_n$. Since $G\in\mathcal{GI}(\xi)$, there is an $\mathbb{E}_\xi$-triangle $\xymatrix@=2em{G'\ar[r]&I\ar[r]&G\ar@{-->}[r]&}$ with $I\in\mathcal{I}(\xi)$ and $G'\in\mathcal{GI}(\xi)$. It follows from (ET4)$^{\rm op}$ that there is a commutative diagram with $\mathbb{E}_\xi$-triangles
$$\xymatrix@R=0.5cm{G'\ar[r]\ar@{=}[d]&H\ar[r]\ar[d]&M\ar[d]\ar@{-->}[r]&\\
G'\ar[r]&I\ar[r]\ar[d]&G\ar[d]\ar@{-->}[r]&\\
&K\ar@{-->}[d]\ar@{=}[r]&K\ar@{-->}[d]&\\
&&&}$$
Then $H\in\mathcal{I}(\xi)_n$ as $I\in\mathcal{I}(\xi)$ and $K\in\mathcal{I}(\xi)_{n-1}$. Since $G'\in\mathcal{GI}(\xi)\subseteq\mathcal{GI}(\xi)_m\subseteq (^\perp\mathcal{I}(\xi)_m\cap\mathcal{I}(\xi)_n)^\perp$ by $(*)$ and $(^\perp\mathcal{I}(\xi)_m\cap\mathcal{I}(\xi)_n)^\perp$ is closed under extensions, it follows that $H\in(^\perp\mathcal{I}(\xi)_m\cap\mathcal{I}(\xi)_n)^\perp$. Thus
$H\in (^\perp\mathcal{I}(\xi)_m\cap\mathcal{I}(\xi)_n)^\perp\cap\mathcal{I}(\xi)_n=\mathcal{I}(\xi)_m$ by Lemma \ref{lemma4.11}, and $M\in\mathcal{GI}(\xi)_m$ by Theorem \ref{thm3}. Therefore, we have proved that $(^\perp\mathcal{I}(\xi)_m\cap\mathcal{I}(\xi)_n, \mathcal{GI}(\xi)_m)$ is a cotorsion pair in $(\mathcal{GI}(\xi)_n, \mathbb{E}_{\mathcal{GI}(\xi)_n}, \mathfrak{s}_{\mathcal{GI}(\xi)_n})$.

Next we claim that it is complete in $(\mathcal{GI}(\xi)_n, \mathbb{E}_{\mathcal{GI}(\xi)_n}, \mathfrak{s}_{\mathcal{GI}(\xi)_n})$. For any $M\in\mathcal{GI}(\xi)_n$, there is an
$\mathbb{E}_\xi$-triangle $$\xymatrix@=2em{G\ar[r]&L\ar[r]&M\ar@{-->}[r]&}$$ with $L\in\mathcal{I}(\xi)_{n}$ and $G\in\mathcal{GI}(\xi)$ by Theorem \ref{thm3}.
Since $(^\perp\mathcal{I}(\xi)_m\cap\mathcal{GI}(\xi)_n, \mathcal{I}(\xi)_m)$ is a complete cotorsion pair in $(\mathcal{GI}(\xi)_n, \mathbb{E}_{\mathcal{GI}(\xi)_n}, \mathfrak{s}_{\mathcal{GI}(\xi)_n})$ by Corollary \ref{cor4.11}, there is an $\mathbb{E}_\xi$-triangle
$$\xymatrix@=2em{N\ar[r]&K\ar[r]&L\ar@{-->}[r]&}$$ with $K\in{^\perp\mathcal{I}}(\xi)_m\cap\mathcal{GI}(\xi)_n$ and $N\in\mathcal{I}(\xi)_m$. It follows from (ET4)$^{\rm op}$ that there is a commutative diagram with $\mathbb{E}_\xi$-triangles
$$\xymatrix@R0.5cm{N\ar[r]\ar@{=}[d]&H\ar[r]\ar[d]&G\ar[d]\ar@{-->}[r]&\\
N\ar[r]&K\ar[r]\ar[d]&L\ar[d]\ar@{-->}[r]&\\
&M\ar@{-->}[d]\ar@{=}[r]&M\ar@{-->}[d]&\\
&&&}$$
Since $N\in\mathcal{I}(\xi)_n$ and $L\in\mathcal{I}(\xi)_n$, it follows that $K\in\mathcal{I}(\xi)_n$, and hence $K\in{^\perp\mathcal{I}}(\xi)_m\cap\mathcal{I}(\xi)_n$.  Since $G\in\mathcal{GI}(\xi)\subseteq\mathcal{GI}(\xi)_m$ and $N\in\mathcal{I}(\xi)_m\subseteq\mathcal{GI}(\xi)_m$, it follows that $H\in\mathcal{GI}(\xi)_m$ because $\mathcal{GI}(\xi)_m$ is closed under extensions by Lemma \ref{lem4.4}. Thus we get an $\mathbb{E}_\xi$-triangle
$$\xymatrix@=2em{H\ar[r]&K\ar[r]&M\ar@{-->}[r]&}$$ with $H\in\mathcal{GI}(\xi)_{m}$ and $K\in{^\perp\mathcal{I}}(\xi)_m\cap\mathcal{I}(\xi)_n$. Therefore, $(^\perp\mathcal{I}(\xi)_m\cap\mathcal{I}(\xi)_n, \mathcal{GI}(\xi)_m)$ is a complete cotorsion pair in the extriangulated category $(\mathcal{GI}(\xi)_n, \mathbb{E}_{\mathcal{GI}(\xi)_n}, \mathfrak{s}_{\mathcal{GI}(\xi)_n})$ by Lemma \ref{lemma2.18}.

Note that $\mathcal{GI}(\xi)_m$ is closed under cones of inflations in  $(\mathcal{B}, \mathbb{E}_{\xi}, \mathfrak{s}_{\xi})$ by Lemma \ref{lem4.4}. Then $\mathcal{GI}(\xi)_m$ is closed under cones of inflations in $(\mathcal{GI}(\xi)_n, \mathbb{E}_{\mathcal{GI}(\xi)_n}, \mathfrak{s}_{\mathcal{GI}(\xi)_n})$. It follows from Lemma \ref{lem2.18} that the complete cotorsion pair $(^\perp\mathcal{I}(\xi)_m\cap\mathcal{I}(\xi)_n, \mathcal{GI}(\xi)_m)$ is hereditary in  $(\mathcal{GI}(\xi)_n, \mathbb{E}_{\mathcal{GI}(\xi)_n}, \mathfrak{s}_{\mathcal{GI}(\xi)_n})$. It is easy to check that $\mathcal{GI}(\xi)_m\cap\mathcal{I}(\xi)_n=\mathcal{I}(\xi)_m$ by the dual version of \cite[Proposition 5.4]{HZZ}, hence the core of this cotorsion pair is $^\perp\mathcal{I}(\xi)_m\cap\mathcal{I}(\xi)_m$.

(2) First we claim that $\mathcal{I}(\xi)_n$ is thick in $(\mathcal{GI}(\xi)_n, \mathbb{E}_{\mathcal{GI}(\xi)_n}, \mathfrak{s}_{\mathcal{GI}(\xi)_n})$. Let $\xymatrix@=2em{X\ar[r]&Y\ar[r]&Z\ar@{-->}[r]&}$ be an $\mathbb{E}_\xi$-triangle in $\mathcal{GI}(\xi)_n$. It suffices to prove that if $Y, Z\in\mathcal{I}(\xi)_n$, then $Z\in\mathcal{I}(\xi)_n$. It follows from Corollary \ref{cor4.7} that $^\perp\mathcal{GI}(\xi)\cap \mathcal{GI}(\xi)_k=\mathcal{I}(\xi)_k$ for any non-negative integer $k$. Hence $Z\in\mathcal{I}(\xi)_{n+1}\cap\mathcal{GI}(\xi)_n={^\perp\mathcal{GI}}(\xi)\cap \mathcal{GI}(\xi)_{n+1}\cap\mathcal{GI}(\xi)_n={^\perp\mathcal{GI}}(\xi)\cap \mathcal{GI}(\xi)_{n}=\mathcal{I}(\xi)_n$. It follows from Corollary \ref{cor4.11} and (1) that $(^\perp\mathcal{I}(\xi)_m\cap\mathcal{GI}(\xi)_n, \mathcal{I}(\xi)_n, \mathcal{GI}(\xi)_m)$ is a hereditary Hovey triple in the extriangulated category $(\mathcal{GI}(\xi)_n, \mathbb{E}_{\mathcal{GI}(\xi)_n}, \mathfrak{s}_{\mathcal{GI}(\xi)_n})$.

Let $m=0$. It is obvious that $(^\perp\mathcal{I}(\xi), \mathcal{I}(\xi))$ is a complete and hereditary cotorsion pair, hence the triple  $(\mathcal{GI}(\xi)_n, \mathcal{I}(\xi)_n, \mathcal{GI}(\xi))$ is a hereditary Hovey triple in the extriangulated category $(\mathcal{GI}(\xi)_n, \mathbb{E}_{\mathcal{GI}(\xi)_n}, \mathfrak{s}_{\mathcal{GI}(\xi)_n})$ by (1). So the corresponding homotopy category is

\hspace{3.4cm}$\underline{^\perp\mathcal{I}(\xi)_m\cap\mathcal{GI}(\xi)_m}\cong \underline{\mathcal{GI}(\xi)}$.
\end{proof}

We end this paper with the proof of Theorem \ref{thm4.15'} in the introduction.
\vspace{2mm}

{\bf Proof of Theorem \ref{thm4.15'}.} The statement is an immediate consequence of Theorem \ref{thm4.15} combined with Proposition \ref{lem4.7}(1). \hfill$\Box$

\renewcommand\refname
\end{document}